\documentclass[
  preprint, 3p, 
  number, 
]{elsarticle}

\journal{Journal of Computational Physics}

\usepackage{graphicx}
\usepackage[reqno]{amsmath}
\usepackage{amssymb}
\let\vec\relax 
\usepackage{epstopdf}
\usepackage{amsthm}
\let\vec\relax 
\usepackage[]{graphicx}               
\usepackage{bm}
\usepackage{esvect}
\usepackage{ulem}
\usepackage{adjustbox}
\usepackage{accents}
\usepackage{stmaryrd}
\usepackage{xcolor}
\usepackage{enumitem}
\usepackage[plainpages=false,pdfpagelabels,hidelinks,unicode]{hyperref} 
\usepackage{booktabs}
\usepackage{enumitem}
\usepackage{pgfplots}
\pgfplotsset{compat=1.14}
\usepgfplotslibrary{groupplots} 

\usepackage{xargs}
\usepackage[colorinlistoftodos,prependcaption,textsize=tiny]{todonotes}
\newcommandx{\unsure}[2][1=]{\todo[linecolor=blue,backgroundcolor=blue!25,bordercolor=blue,#1]{#2}}
\newcommandx{\remarkBox}[2][1=]{\todo[linecolor=red,backgroundcolor=red!25,bordercolor=red,#1]{#2}}

\newtheorem{thm}{Theorem}
\newtheorem{lem}{Lemma}
\newtheorem{rem}{Remark}

\newtheorem{ex}{Example}

\begin{document}


\definecolor{chcol}{rgb}{0.4,0.,0.9}
\newcommand{\change}[1]{\textcolor{chcol}{#1}}
\newcommand{\viscosity}{\mu}     
\renewcommand{\Pr}{\text{Pr}}     
\renewcommand{\Re}{\text{Re}}   
\newcommand{\Ma}{\text{Ma}}     
\newcommand{\halfComma}{\kern 0.083334em}
\newcommand{\pderivative}[2]{\frac{\partial #1}{\partial #2}}
\newcommand{\avg}[1]{\left\{\hspace*{-3pt}\left\{#1\right\}\hspace*{-3pt}\right\}}
\newcommand{\jump}[1]{\ensuremath{\left\llbracket #1 \right\rrbracket}}
\newcommand{\shat}{\ensuremath{\hat{s}}}        
\newcommand{\dS}{{\,\operatorname{dS}}}         
\newcommand{\PBT}{{\,\operatorname{PBT}}}    
\newcommand{\ec}{{\mathrm{EC}}}                     
\newcommand{\es}{{\mathrm{ES}}}                     
\newcommand{\ent}{{S}}                                     
\newcommand{\ma}{-}                                        
\renewcommand{\sl}{+}                                      
\newcommand{\dx}{\,\text{d}x}                           
\newcommand\iprod[1]{\left\langle #1\right\rangle}                                             
\newcommand\inorm[1]{\left |\left| #1\right|\right|}                                               
\newcommand\vnorm[1]{\left| #1\right|}                                              		       
\newcommand\iprodN[1]{\left\langle #1\right\rangle_{\!N}}                                 
\newcommand\inormN[1]{\left |\left| #1\right|\right|_{N}}                                     
\newcommand\irefInt{\int\limits_{-1}^{1}}                                                            
\newcommand\ivolN{\int\limits_{ N }\! }                                                              
\newcommand\isurfE{\int\limits_{\partial E }\! }                                                   
\newcommand\volEN{\mkern-11mu\int\limits_{E , N}\mkern-5mu }                    
\newcommand\isurfEN{\mkern-11mu\int\limits_{\partial E , N}\mkern-11mu }    
\renewcommand\vec[1]{\accentset{\,\rightarrow}{#1}}                              
\newcommand\spacevec[1]{\accentset{\,\rightarrow}{#1}}                        
\newcommand\contravec[1]{\tilde{ #1}}                                                     
\newcommand\contraspacevec[1]{\spacevec{\tilde{#1}}}                          
\newcommand\spacevecg[1]{\vv #1}                                                         
\newcommand\spacestatevec[1]{\vv{\spacevec{#1}}}                               
\newcommand\statevec[1]{\mathbf #1}                                                     
\newcommand\statevecGreek[1]{\boldsymbol #1}                                     
\newcommand\contrastatevec[1]{\tilde{\mathbf #1}}                                 
\newcommand\acclrvec[1]{\accentset{\,\leftrightarrow}{#1}}                      
\newcommand\bigstatevec[1]{\acclrvec{{\mathbf #1}}}                              
\newcommand\biggreekstatevec[1]{\acclrvec{\boldsymbol #1}}                
\newcommand\bigcontravec[1]{\acclrvec{\tilde{\mathbf{#1}}}}                   
\newcommand\biggreekcontravec[1]{\acclrvec{\tilde{\boldsymbol #1}}}    
\newcommand\vecNabla{\accentset{\,\rightarrow}\nabla}                         
\newcommand\vecNablaXi{\accentset{\,\rightarrow}\nabla_{\!\xi}}            
\newcommand\vecNablaX{\accentset{\,\rightarrow}\nabla_{\!x}}              
\newcommand\mmatrix[1]{\underbar{#1}}				
\newcommand\fiveMatrix[1]{\mathsf{ #1}}                          
\newcommand\fifteenMatrix[1]{\underline{\mathsf{ #1}}}    
\newcommand\matrixvec[1]{\mathcal #1}                           
\newcommand\bigmatrix[1]{\mathfrak #1}                          
\newcommand{\dmat}{\matrixvec{D}}     
\newcommand{\qmat}{\matrixvec{Q}}    
\newcommand{\mmat}{\matrixvec{M}}   
\newcommand{\bmat}{\matrixvec{B}}    
\newcommand\interiorfaces{\genfrac{}{}{0pt}{}{\mathrm{interior}}{\mathrm{faces}}}
\newcommand\boundaryfaces{\genfrac{}{}{0pt}{}{\mathrm{boundary}}{\mathrm{faces}}}
\newcommand\allfaces{\genfrac{}{}{0pt}{}{\mathrm{all}}{\mathrm{faces}}}
\newcommand{\testfuncOne}{\statevecGreek{\varphi}}  
\newcommand{\testfuncTwo}{\biggreekstatevec{\vartheta}}
\newcommand{\testfuncPhi}{\boldsymbol\phi}
%
\newcommand{\DD}{\spacevec{\mathbb{D}}\cdot}
\newcommand{\DDs}{\spacevec{\mathbb{D}}^{s}\cdot}
\newcommand{\IN}[1]{\mathbb I^{N}\!\!\left(#1\right)} 
\newcommand{\PN}[1]{\mathbb P^{#1}}
\newcommand{\LTwo}[1]{\mathbb L^{2}\!\left(#1\right)}
%
%
\newcommand\overRe{\frac{1}{{\operatorname{Re} }}}
\newcommand\twooverRe{\frac{2}{{\operatorname{Re} }}}
%
%
\newcommand\oneHalf{\frac{1}{2}}
\newcommand\oneFourth{\frac{1}{4}}
\makeatletter
\providecommand{\doi}[1]{%
  \begingroup
    \let\bibinfo\@secondoftwo
    \urlstyle{rm}%
    \href{http://dx.doi.org/#1}{%
      doi:\discretionary{}{}{}%
      \nolinkurl{#1}%
    }%
  \endgroup
}
\makeatother

\begin{frontmatter}

\title{Energy Bounds for Discontinuous Galerkin Spectral Element Approximations of Well-Posed Overset Grid Problems for Hyperbolic Systems}

\author[1,2]{David A. Kopriva}
\ead{dkopriva@fsu.edu}
\address[1]{Department of Mathematics, The Florida State University, Tallahassee, FL 32306, USA}
\address[2]{Computational Science Research Center, San Diego State University, San Diego, CA, USA}

\author[3]{Andrew R. Winters\corref{cor2}}
\ead{andrew.winters@liu.se}
\address[3]{Department of Mathematics, Applied Mathematics, Link\"{o}ping University, 581 83 Link\"{o}ping, Sweden}

\author[3,4]{Jan Nordstr\"om}
\ead{jan.nordstrom@liu.se}
\address[4]{Department of Mathematics and Applied Mathematics
 University of Johannesburg
 P.O. Box 524, Auckland Park 2006, South Africa.}
 \cortext[cor2]{Corresponding author}
 
\begin{abstract}
We show that even though the Discontinuous Galerkin Spectral Element Method is stable for hyperbolic boundary-value problems, and the overset domain problem is well-posed in an appropriate norm, the energy of the approximation of the latter is bounded by data only for fixed polynomial order, mesh, and time. In the absence of dissipation, coupling of the overlapping domains is destabilizing by allowing positive eigenvalues in the system to be integrated in time. This coupling can be stabilized in one space dimension by using the upwind numerical flux. To help provide additional dissipation, we introduce a novel penalty method that applies dissipation at arbitrary points within the overlap region and depends only on the difference between the solutions. We present numerical experiments in one space dimension to illustrate the implementation of the well-posed penalty formulation, and show spectral convergence of the approximations when sufficient dissipation is applied.
\end{abstract}


\end{frontmatter}

\section{Introduction}

Overset grid methods \cite{Chan99,Meakin:1999io} accommodate complex geometries by overlapping simpler geometry fitted grids. Stability of the procedures used to couple the grids has long been a practical and theoretical issue \cite{SHARAN2018199}, especially at high order, and to date full stability proofs are not available. One of the difficulties in proving stability is that the underlying partial differential equation (PDE) described on the overset domains is not necessarily well-posed -- a necessary condition for convergence of a numerical scheme -- even when dissipative (characteristic) boundary conditions are applied at the grid boundaries. Well-posed formulations for linear systems in one or two space dimensions that use penalties along the overlap boundaries were presented in \cite{KOPRIVA2022110732}. In this paper, we simplify the analysis by considering only the one dimensional problem, while noting the differences with higher dimensions.

To set the stage in its simplest form, the problem to be solved in one space dimension for the scalar initial-boundary value problem (IBVP) is
\begin{equation}
OP\;\left\{\begin{gathered}
\omega_t + \alpha \omega_x = 0, \; \alpha = {\rm const}>0,\quad x\in\Omega = [a,d]\hfill\\
\omega(a,t) = 0\hfill\\
\omega(x,0) = \omega_0(x), \hfill
\end{gathered} \right.
\label{eq:OPScalar}
\end{equation}
which we denote as the ``Original Problem" (OP). The overset domain problem
splits the domain $\Omega$ into two overlapping subdomains  $\Omega_u = [a,c]$, $\Omega_v = [b,d]$, and defines the Base (B) and Overset (O) problems
\begin{equation}
B\;\left\{\begin{gathered}
u_t + \alpha u_x = 0, \quad x\in\Omega_u\hfill\\
u(a,t) = 0\hfill\\
u(x,0) = \omega_0(x) \hfill
\end{gathered} \right. \quad
O\;\left\{\begin{gathered}
v_t + \alpha v_x = 0, \quad x\in\Omega_v\hfill\\
v(b,t) = u(b,t)\hfill\\
v(x,0) = \omega_0(x) \hfill
\end{gathered} \right. .
\label{eq:L-RSystem}
\end{equation}
The subdomains,  overlap so that $\Omega = \Omega_u \cup \Omega_v$, see Fig.~\ref{fig:Overlap1D}. We will call \eqref{eq:L-RSystem} and its equivalent hyperbolic system version the {\it characteristic form of the overset domain problem} because it directly specifies characteristic boundary conditions on the problem $O$.  We choose the boundary condition on the left at $x=a$ to be zero so that the analytic solution is bounded solely by the initial condition; no energy is added over time through the left boundary.

\begin{figure}[tbp] 
   \centering
   \includegraphics[width=4in]{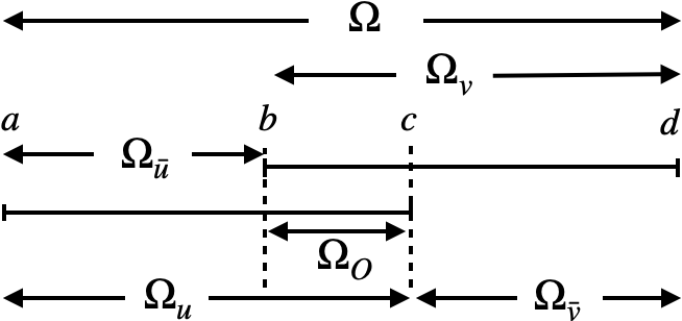}
   \caption{Overset domain definitions in 1D}
   \label{fig:Overlap1D}
\end{figure}

It was shown in \cite{KOPRIVA2022110732} that the characteristic formulation for the scalar problem, \eqref{eq:L-RSystem}, is well-posed.
By standard energy arguments \cite{Nordstrom:2017yu}, the $\mathbb L^2$ norm of $u$ is bounded over the entire domain,
\begin{equation}
    \inorm{u(T)}^2_{\Omega_u}+\alpha\int_0^T u^2(c,t)\,\mathrm{d}t \le \inorm{\omega_0}^2_{\Omega_u},
\end{equation}
where $\iprod{u,w}_{\Omega_u} = \int_{\Omega_u} uw\,\mathrm{d}{\Omega_u}$ and $\inorm{u}^2_{\Omega_u }= \iprod{u,u}_{\Omega_u}$.
The overset problem, $O$, has a boundary condition that depends on $u$, for
\begin{equation}
\frac{d}{dt}\inorm{v}_{\Omega_v}^2 = -\alpha v^2(d,t) + \alpha u^2(b,t),
\label{eq:VEnergyWithUBC}
\end{equation}
so the energy for $v$ can be bounded by the initial data only if $u(b,t)$ can be bounded by the initial data (since the boundary data for $u=0$). In the continuous case, we can bound the value of $u^2(b,t)$ with integration-by-parts of the PDE over the complementary domain, $\Omega_{\bar u} = [a,b]$,
\begin{equation}
\alpha u^2(b,t) = -\frac{d}{dt}\inorm{u}_{\Omega_{\bar u}}^2.
\label{eq:ubarvalue}
\end{equation}
This was known as ``Technique 1" in \cite{KOPRIVA2022110732}. In fact, the following was proved for the scalar problem \cite{KOPRIVA2022110732}:
\begin{equation}
\begin{split}
& \inorm{u}_{\Omega_u}^2 \le \inorm{\omega_0}_{\Omega_u}^2
\\&
\inorm{v}_{\Omega_v}^2 \le \inorm{\omega_0}_{\Omega_v}^2 +  \inorm{\omega_0}_{\Omega_{\bar u}}^2 = \inorm{\omega_0}^2_{\Omega}.
\end{split}
\label{eq:Sharp1DScalarEstimate}
\end{equation}
Thus, the total energy in $u$ is bounded by its initial energy, while the energy in $v$ is bounded by the initial total energy in the full domain.
\textcolor{black}{That, uniqueness, and the fact that the solutions of \eqref{eq:L-RSystem} are the same in norm to the solution of OP, which is known to be well-posed, imply existence, from which well-posedness of \eqref{eq:L-RSystem} follows. See \cite{KOPRIVA2022110732}. }
\begin{rem}
In the overset domain problem, boundary data must come from a donor subdomain, e.g. $u^2(b,t)$ in \eqref{eq:VEnergyWithUBC}. Those boundary values do not have specified data. Instead, Technique 1 reverses the usual energy method process, which converts volume quantities to boundary values at which data is applied, and instead converts boundary values to volume quantities, which can then be connected to the initial conditions, e.g. \eqref{eq:ubarvalue}. Technique 1 represents the fact that for the PDE on any subdomain of the original, the energy is influenced only by initial data or data at the subdomain boundaries.
\end{rem}
\begin{rem}
The scalar problem is a special case in that there is no feedback from the overset subdomain, $\Omega_v$, to the base subdomain, $\Omega_u$. The base subdomain acts as an independent problem, decoupled from the overset subdomain. The scalar problem does not describe the full set of behaviors possible in systems of equations, such as reflections at the boundaries of the overlap domain, $\Omega_O$.
\end{rem}

A result similar to \eqref{eq:Sharp1DScalarEstimate} is found for symmetric hyperbolic systems
\begin{equation}
\statevec u_t + \mmatrix A\statevec u_x = 0
\end{equation}
with homogeneous characteristic boundary conditions when $\mmatrix A = \mmatrix A^T = \rm const$,
if one diagonalizes the system first before computing the energy estimate. Diagonalizing allows one to split the characteristics so that the complementary domain energies for ``incoming" waves are decoupled. In this way, coupling within the overlap $\Omega_O = [b,c]$ only occurs along characteristics. 

Unfortunately, decoupling through diagonalization is usually available only in one space dimension and for constant coefficients. Even for constant coefficient systems, the coefficient matrices are not typically simultaneously diagonalizable in more than one space dimension. 

If diagonalization is not done (or is not possible), then one gets energy bounds \cite{KOPRIVA2022110732} from
\begin{equation}\begin{split}
&\frac{d}{dt}\inorm{\statevec u}_{\Omega_u}^2
=  -\left\{\statevec u^T(a)\left|\mmatrix A^-\right|\statevec u(a)
+ \statevec u^T(c)\mmatrix A^+\statevec u(c) \right\}
+\statevec v^T(c)\left|\mmatrix A^-\right|\statevec v(c)
\\&
\frac{d}{dt}\inorm{\statevec v}_{\Omega_v}^2
=
-\left\{\statevec v^T(b)\left|\mmatrix A^-\right|\statevec v(b)
+ \statevec v^T(d)\mmatrix A^+\statevec v(d) \right\}
+\statevec u^T(b)\mmatrix A^+\statevec u(b) ,
\end{split}
\label{eq:EnergyInEachDomain}
\end{equation}
where $\mmatrix A^\pm = \oneHalf(\mmatrix A \pm |\mmatrix A|)$ are matrices whose eigenvalues have a single sign.
Note that the terms in braces on the right of both equations in \eqref{eq:EnergyInEachDomain} are non-negative, and correspond to energy loss due to energy leaving each subdomain. The remaining term in each is due to energy coming in from the donor subdomain.
Adding the two equations in \eqref{eq:EnergyInEachDomain},
\begin{equation}
\frac{d}{dt}\inorm{\statevec u}_{\Omega_u}^2  + \frac{d}{dt}\inorm{\statevec v}_{\Omega_v}^2\le \statevec v^T(c)\left|\mmatrix A^-\right|\statevec v(c) + \statevec u^T(b)\mmatrix A^+\statevec u(b).
\label{eq:EqnWithGrowthTerms}
\end{equation}
The quantities on the right are non-negative and so imply energy growth in each subdomain. Worse, they are not bounded by data unless additional information about the behavior of the solution in the overlap region, or equivalently the complementary domains, $\Omega_{\bar u}$ and $\Omega_{\bar v}$, can be provided. We see, then, that the overset domain problem is not well-posed in $\mathbb L^2(a,d)$ using the energy norm represented by the sum on the left. That norm is also not consistent with the energy norm of the OP because it double counts the energy in the overlap region.

In \cite{KOPRIVA2022110732}, a novel well-posed penalty formulation of the overset domain problem for systems was proposed. For {$\mathbb L^2$ test functions $\statevecGreek\phi$ and $\statevecGreek\varphi$, the weak form of the penalty formulation is
\begin{equation}
\begin{gathered}
 \iprod{\statevecGreek\phi,\statevec u_t}_{\Omega_u} + \iprod{\statevecGreek\phi,\mmatrix A u_x }_{\Omega_u}+
 \left.\statevecGreek\phi^T\ \mmatrix{$\Sigma$}_u^b (\statevec u-\statevec v)\right|_{b^+}+
 \left.\statevecGreek\phi^T\ \mmatrix{$\Sigma$}_u^c (\statevec u-\statevec v)\right|_{c^-}= 0,\hfill\\
\iprod{\statevecGreek\varphi,\statevec v_t}_{\Omega_v} + \iprod{\statevecGreek\varphi,\mmatrix A \statevec v_x}_{\Omega_v} +  \left.\statevecGreek\varphi^T \mmatrix{$\Sigma$}_v^b (\statevec v-\statevec u)\right|_{b^+}+   \left.\statevecGreek\varphi^T \mmatrix{$\Sigma$}_v^c (\statevec v-\statevec u)\right|_{c^-}= 0,\hfill\\
\end{gathered}
\label{eq:WeakFormsonFullDomains}
\end{equation}
plus dissipative physical boundary conditions. The superscripts $\pm$ refer to the limit from the right and the left. The penalty formulation is well posed in $\mathbb L^2(a,d)$ using the {\it Overset Domain Norm}
\begin{equation}
\begin{split}
E(\statevec u,\statevec v)
&\equiv \inorm{\statevec u}^2_{\Omega_u} + \inorm{\statevec v}^2_{\Omega_v}
 -  \left\{ \eta \inorm{\statevec u}^2_{\Omega_O} +  (1-\eta)\inorm{\statevec v}^2_{\Omega_O}\right\}
\\&
=\inorm{\statevec u}^2_{\Omega_{\bar u}} + (1-\eta)\inorm{\statevec u}^2_{\Omega_O} + \inorm{\statevec v}^2_{\Omega_{\bar v}} + \eta \inorm{\statevec v}^2_{\Omega_O}
\end{split}
\label{eq:EnergyNorm}
\end{equation}
for any $0<\eta<1$, when the symmetric, positive penalty matrices $\mmatrix{$\Sigma$}$ satisfy
\begin{subequations}
\begin{align}
\beta A + \mmatrix{$\Sigma$}_u =  \mmatrix{$\Sigma$}_v ,
\label{eq:1DSigmaMatrixConditionsW2i}
\\
2 \mmatrix{$\Sigma$}_v - \beta A \ge 0,
\label{eq:1DSigmaMatrixConditionsW3i}
\end{align}
\label{eq:1DSigmaMatrixConditionsWi}
\end{subequations}
where $\beta = \eta$ at $x=b$ and $\beta = (1-\eta)$ at $x=c$. This means that the solution of \eqref{eq:WeakFormsonFullDomains} is unique, and for initial conditions matching that of the OP, the solutions match that of the OP in the norm \eqref{eq:EnergyNorm}.
The proof depends on the fact that integration-by-parts applies on any subinterval of the domain, i.e., ``Technique 1". The penalty terms in \eqref{eq:EnergyNorm} serve to remove the growth terms in the bound \eqref{eq:EqnWithGrowthTerms}.

Unfortunately, Technique 1 does not hold discretely for all but the most trivial approximations and situations. High order methods, especially, can have a larger domain of dependence than the PDE over a time step. Instead, at best, an approximation has a summation-by-parts (SBP) property only on a whole subdomain (or element). This means that one cannot follow the steps that show the continuous problem is well-posed with the equivalent discrete argument to show stability, as is typically done with schemes that satisfy a SBP property \cite{Nordstrom:2017yu}.

So the question is: ``What can one say about the stability of a numerical scheme that does not have full access to the integration-by-parts properties of the original PDE when applied to an overset grid problem?" We study this question in the context of discontinuous Galerkin spectral element (DGSEM) approximations of the overset grid problem for scalar and systems of linear equations in one space dimension by deriving energy bounds and the eigenvalue structure for each.

We target the DGSEM because:
\begin{enumerate}
\item The DGSEM is of arbitrary order depending on the order of the approximating polynomials and quadrature, and so the results count as high order approximations.
\item The DGSEM satisfies the SBP property and has been shown to be stable for the original IBVP \eqref{eq:OPScalar} and its system equivalent \cite{Gassner:2013ol,Gassner:2013ij}. There is also a relation to high order SBP finite difference methods \cite{Gassner:2013ol}.
\item The polynomial approximations used in the DGSEM define interpolations intrinsically, which bypasses the need for ad hoc interpolation schemes to find solution values between grid points.
\item The DGSEM approximates a weak form \eqref{eq:WeakFormsonFullDomains} of the equations, so there is no need for specially designed lifting operators to compute penalties at non-mesh points that appear in the strong formulation.
\end{enumerate}

In this paper, we show that, even though the DGSEM is stable on the OP, and the overset domain problem is well-posed in an appropriate norm, the energy of the approximation is bounded by data only for fixed polynomial order, mesh, and time for the overset problem using either characteristic boundary conditions or the well-posed formulation \eqref{eq:WeakFormsonFullDomains} derived in \cite{KOPRIVA2022110732}. To provide additional dissipation to counteract growth, we implement a novel penalty method introduced in \cite{KOPRIVA2022110732} that applies dissipation at arbitrary points within the overlap region and depends on the difference between the solutions. Finally, we present numerical experiments in one space dimension to illustrate how the DGSEM approximation performs with the penalty formulation of \cite{KOPRIVA2022110732}.

\section{DGSEM Approximations to the Overset Grid Problem}

We present the DGSEM approximation of two formulations of the overset grid problem. The first is the characteristic formulation, for both scalar and systems of equations, where upwind data needed at the subdomain boundaries is taken from the donor subdomain. As shown in \cite{KOPRIVA2022110732}, this approach is well-posed for the scalar problem, but not, in general, for the system. The second formulation is the penalty formulation \eqref{eq:WeakFormsonFullDomains} of \cite{KOPRIVA2022110732}, which is well-posed in a suitable norm.

\subsection{The DGSEM Approximation of the Scalar Characteristic Problem}
We begin with the approximation of the scalar problem, \eqref{eq:L-RSystem}, which is a common, simple, starting point. As noted in the introduction, the scalar problem is well-posed, but it does not need coupling of the overset domain back to the base domain. As a result, we will see that the energy bounds obtained are better than in the general system case.

The description of the DGSEM is now standard \cite{Kopriva:2009nx,Winters2021} and won't be repeated in detail. The DGSEM approximates the weak forms of the equations \eqref{eq:L-RSystem}. Let $\phi$, $\varphi$ be test functions. Then the weak forms of the PDEs become
\begin{equation}
\begin{gathered}
\iprod{u_t,\phi}_{\Omega_u} +\alpha\iprod{u_x,\phi}_{\Omega_u} = 0 \\
\iprod{v_t,\varphi}_{\Omega_v} +\alpha\iprod{v_x,\varphi} _{\Omega_v}= 0.
\end{gathered}
\end{equation}

For the numerical approximation, the domains are subdivided into elements, and each element is mapped onto the standard interval $[-1,1]$. Solutions are approximated by polynomials of degree $N$ and inner-products/integrals are approximated by Legendre-Gauss-Lobatto (LGL) quadratures.  Elements are coupled by way of a two-point numerical flux $F^*(\cdot,\cdot)$.

Let the set $\underline{e}^k = [x_{k-1},x_k],\; k = 1,\ldots,K_u$, where $x_k = a + k\Delta x_k, \; k=0,\ldots,K_u$ be the grid of elements covering the base subdomain, $\Omega_u$. For convenience, we assume that the elements have the same size, and that the elements are consecutively ordered, though neither is necessary in practice.

We map each element to the reference element $E=[-1,1]$ with the affine map
\begin{equation}
x = x_{k-1} + \frac{\xi+1}{2}\left( x_k - x_{k-1}\right) = x_{k-1} + \frac{\xi+1}{2}\Delta x_k,
\end{equation}
so that on the reference element,
\begin{equation}
\frac{\Delta x_k}{2} \iprod{u_t,\phi}_{E}  + \alpha\iprod{u_\xi,\phi}_{E}= 0.
\end{equation}
From here, we drop the subscripts, $E$ and $k$. Integrating by parts,
\begin{equation}
\frac{\Delta x}{2} \iprod{u_t,\phi}  +\phi f |_{-1}^1- \alpha\iprod{u,\phi_\xi}= 0,
\end{equation}
where $f = \alpha u$.

Finally, we replace functions by their polynomial approximations, integrals with quadrature, and the surface (boundary) flux by the numerical flux, $ F^*$.
To those ends, let us define $U$ to be a polynomial of degree $\le N$ defined on the interval $\xi\in [-1,1]$. In nodal representation,
\begin{equation}
U(\xi,t) = \sum_{j=0}^N U_j(t) \ell_j(\xi),
\end{equation}
where $\ell_j$ is the $j^{\mathrm{th}}$ Lagrange interpolating polynomial of degree $N$, and $U_j$ is the associated nodal value at the $j^{\mathrm{th}}$ LGL quadrature point.
Next, let the discrete inner product be
\begin{equation}
\iprod{U,W}_N =\sum_{j=0}^N U_j W_j w_j,
\label{eq:DiscreteNorm}
\end{equation}
where $w_j$ is the LGL quadrature weight associated with the nodal point $x_j$. The discrete inner product induces the discrete $\mathbb L^2$ norm, $\inormN{U}^2 = \iprodN{U,U}$. Finally, for given left and right states $U^L$ and $U^R$, the upwind flux is $F^*(U^L,U^R) = \alpha U^L$ for $\alpha>0$.

Substituting the approximations gives the DGSEM for the scalar problem on an element of the base domain
\begin{equation}
\frac{\Delta x}{2} \iprod{U_t,\phi}_N  +\phi F^* |_{-1}^1- \alpha\iprod{U,\phi_\xi}_N= 0,
\label{eq:UDGSEMWeakForm}
\end{equation}
with the external state at the left physical boundary taken to be zero. In general we do not need to specify the specific element $k$ on which the solution is considered, e.g. $U^k$, and will do so only if necessary.
For $v$ on the overset domain, we also have the same weak form,
\begin{equation}
\frac{\Delta x}{2} \iprod{V_t,\varphi}_N  +\varphi F^* |_{-1}^1- \alpha\iprod{V,\varphi_\xi}_N= 0,
\label{eq:VDGSEMWeakForm}
\end{equation}
but with the external state at $x=b$ taken from the base solution polynomial, $U(b,t)$.

\subsection{Approximation of Systems with Characteristic Interface Conditions}

The systems of equations are of the form
\begin{equation}
\statevecGreek \omega_t + \mmatrix{A}\statevecGreek \omega_x = 0,
\label{eq:PDESystem}
\end{equation}
where $\statevecGreek \omega \in \mathbb{R}^m$. For simplicity, we assume that the $m\times m$ matrix $\mmatrix A$ is constant coefficient, and already symmetric, i.e. $\mmatrix{A} = \mmatrix{A}^T$. For physical boundary conditions we choose characteristic boundary conditions with external states $\statevec g_L = \statevec g_R = \statevec 0$, which corresponds to no gain of energy at the physical boundaries.

The natural instinct is to approximate the overset domain problem with characteristic boundary conditions at the non-physical subdomain boundaries. The characteristic formulation is well-posed in one space dimension, but is not in higher dimensions \cite{KOPRIVA2022110732}. Translated to the DGSEM, one starts with the usual approximation, but where the external state for the numerical flux at the subdomain boundaries is taken from the donor mesh/subdomain.

For the system, the DGSEM approximation on the reference element for the base domain is
\begin{equation}
\frac{\Delta x}{2} \iprod{\statevec U_t,\statevecGreek\phi}_N  +\statevecGreek\phi^T\left\{ \statevec F^*- \statevec F \right\} |_{-1}^1+
\iprod{\mmatrix A \statevec U_\xi,\statevecGreek\phi}_N= 0,
\label{eqUDGSEMWeakForm}
\end{equation}
where we define $\statevec F =\mmatrix A\statevec U$ and the inner product as $\iprodN{\statevec F, \statevec G} = \sum_{j=1}^N \statevec F^T_j \statevec G_j w_j$. By SBP, this form for the system is algebraically equivalent to the form used in \eqref{eq:UDGSEMWeakForm} \cite{gassner2010}.

To impose characteristic boundary conditions we use the upwind numerical flux. For any two states $\statevec U^L$ and $\statevec U^R$, the upwind flux chooses the state asccording to the signs of the eigenvalues of $\mmatrix A$ as
\begin{equation}
\statevec F^*(\statevec U^L, \statevec U^R) = \mmatrix A^+\statevec U^L  + \mmatrix A^-\statevec U^R = \oneHalf (\statevec F^L + \statevec F^R) - \oneHalf |\mmatrix A|(\statevec U^R - \statevec U^L) \equiv \avg{\statevec F} - \oneHalf |\mmatrix A|\jump{\statevec U}.
\label{eq:UpwindNumericalFlux}
\end{equation}
The external states used are $\statevec g_L  = \statevec 0$ at $ x=a$ and $\statevec g_R=\statevec 0$ at $x=d$, and
\begin{equation}
\statevec U^R(1,t) = \statevec V(c,t),\quad \statevec V^L(-1,t) = \statevec U(b,t)
\end{equation}
for the elements with the overlap domain boundaries.

\subsection{Approximation of Systems with a Pure Penalty Formulation}

To create the DGSEM approximation for the system, \eqref{eq:WeakFormsonFullDomains}, we again subdivide each domain $\Omega_u,\Omega_v$ into elements $\underline{e}_u^k,\bar{e}_v^k$ (Fig.~\ref{fig:1DElements}).
\begin{figure}[htbp] 
   \centering
   \includegraphics[width=4in]{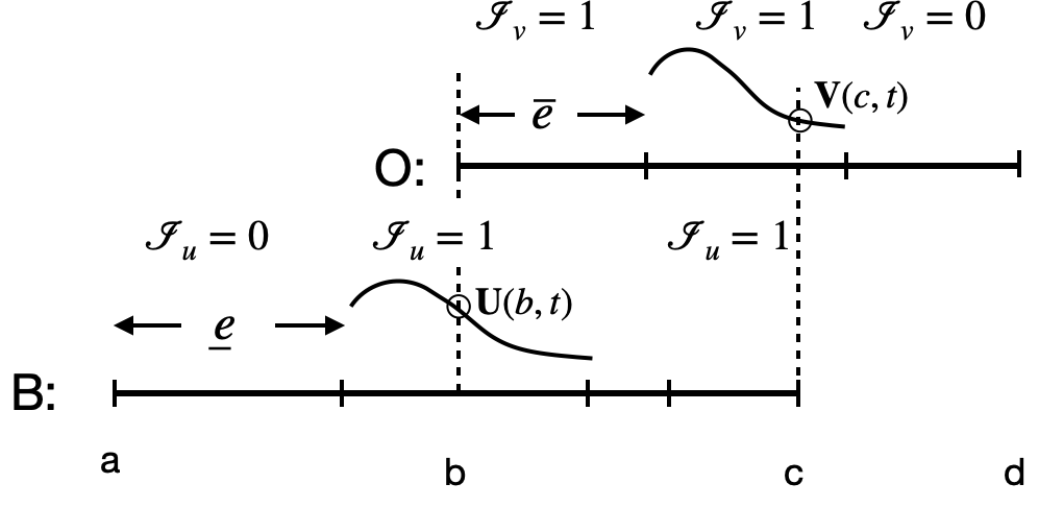}
   \caption{Base (B) overlap (O) subdomains divided into elements.}
   \label{fig:1DElements}
\end{figure}
Associated with each element we define an indicator function that is active only if an overlap boundary falls within it, i.e.
\begin{equation}
\mathcal I^k_u(x) = \left\{\begin{gathered} 1\quad \text{if } x\in \underline{e}^k_u \hfill\\
0\quad \text{otherwise} \hfill\end{gathered} \right.
\end{equation}
and similarly for $\mathcal I^k_v$.
 On the reference element, the elemental contributions become
\begin{equation}
\begin{gathered}
 \iprod{\mathcal J\statevecGreek\phi,\statevec u_t}_E + \iprod{\statevecGreek\phi,\mmatrix A \statevec u_\xi }_{E}+
 \mathcal I^k_u(b) \left.\statevecGreek\phi^T \mmatrix{$\Sigma$}_u^b (\statevec u-\statevec v)\right|_{b^+}+
 \mathcal I^k_u(c) \left.\statevecGreek\phi^T \mmatrix{$\Sigma$}_u^c (\statevec u-\statevec v)\right|_{c^-}= 0,\hfill\\
\iprod{\mathcal J\statevecGreek\varphi,\statevec v_t}_{E} + \iprod{\statevecGreek\varphi,\mmatrix A \statevec v_\xi}_{E}
+   \mathcal I^k_v(b)\left.\statevecGreek\varphi^T \mmatrix{$\Sigma$}_v^b (\statevec v-\statevec u)\right|_{b^+}
+   \mathcal I^k_v(c) \left.\statevecGreek\varphi^T \mmatrix{$\Sigma$}_v^c (\statevec v-\statevec u)\right|_{c^-}= 0\hfill\\
\end{gathered}
\label{eq:WeakFormsFonFullDomainsXform}
\end{equation}
where $\mathcal J = dx/d\xi = \Delta x/2$.

\begin{rem}For clarity in keeping track of locations, we leave the physical space values $b$ and $c$ rather than write them explicitly in the reference space. In practice the reference space locations are located and stored.\end{rem}

In the final approximation within an element, the element boundary fluxes are replaced with a numerical flux except at the overlap interface points, where the boundary coupling is handled through the penalty terms. The result is the DGSEM/overset grid approximation of the penalty formulation
\begin{equation}
\begin{split}
 \iprod{\mathcal J\statevecGreek\phi,\statevec U_t}_N
 &+ \left.\statevecGreek\phi^T\left\{\hat{\statevec F}-\statevec F\right\}\right|^1_{-1} +\iprod{\statevecGreek\phi,\mmatrix A \statevec U_\xi }_{N }
 \\&+
 \mathcal I^k_u(b)  \left.\statevecGreek\phi^T \mmatrix{$\Sigma$}_u^b (\statevec U-\statevec V)\right|_b+
 \mathcal I^k_u(c)  \left.\statevecGreek\phi^T \mmatrix{$\Sigma$}_u^c (\statevec U-\statevec V)\right|_c= 0
 \end{split}
 \label{eq:DGSEMChimeraU}
\end{equation}
and
\begin{equation}
\begin{split}
\iprod{\mathcal J\statevecGreek\varphi,\statevec V_t}_{N }
&+\left.\statevecGreek\varphi^T\left\{\hat{\statevec F}-\statevec F\right\}\right|^1_{-1}
+ \iprod{\statevecGreek\varphi,\mmatrix A \statevec V_\xi}_{N }
\\&+  \mathcal I^k_v(b)  \left.\statevecGreek\varphi^T \mmatrix{$\Sigma$}_v^b (\statevec V-\statevec U)\right|_b
+   \mathcal I^k_v(c)  \left.\statevecGreek\varphi^T \mmatrix{$\Sigma$}_v^c (\statevec V-\statevec U)\right|_c= 0,
 \end{split}
\label {eq:DGSEMChimeraV}
\end{equation}
where $\statevec F(\statevec q) = \mmatrix A\statevec q$ for some $\statevec q$. For the numerical flux,

\begin{equation}
\hat{\statevec F}(\statevec q;x) = (1- \mathcal I^k_q(x))\statevec F^* +  \mathcal I^k_q(x)\statevec F(\statevec q),
\label{eq:ChimeraNumHatFlux}
\end{equation}
where $\statevec F^*\left(\statevec Q^L,\statevec Q^R\right)$ is the usual numerical flux given two element interface values from the left and right, e.g. \eqref{eq:UpwindNumericalFlux}. The flux \eqref{eq:ChimeraNumHatFlux} is just a formal way of saying that the usual numerical flux is used at all element boundaries except at the overlap interface boundaries, where the penalty flux is used instead.

%

\section{Energy Bounds}
Here we derive energy bounds for the characteristic and penalty approximations. Before doing so, we need a few theoretical results about the polynomial approximations.
\subsection{Polynomial Approximation Summary}

To study the stability of the DGSEM approximation to the overset grid problem, we will use a number of results from polynomial approximation theory and spectral element methods. The results we summarize here are found in \cite{CHQZ:2006}.

First, the discrete norm induced by the inner product \eqref{eq:DiscreteNorm} is equivalent to the continuous norm. There exist constants $C_1$ and $C_2$ such that for any polynomial $U$ of degree $\le N$,
\begin{equation}
C_1 \inorm{U}_{\LTwo{-1,1}}\le \inormN{U} \le C_2 \inorm{U}_{\LTwo{-1,1}}
\label{eq:NormEquivalence}
\end{equation}
where $\inorm{U}^2_{\LTwo{-1,1}} = \int_{-1}^{1} U^2\,\mathrm{d}x$ is the usual $\mathbb L^2$ norm.
For the Legendre polynomial approximations, $C_1 = 1$ and $C_2 = \sqrt{3}$ (\cite{CHQZ:2006},  Sec.~5.3).

The continuous norms satisfy the inverse inequality (\cite{CHQZ:2006},  Sec.~5.4.4). For any polynomial $U$ of degree $\le N$,
\begin{equation}
\inorm{U}_{\mathbb L^q(-1,1)}\le C \Delta x^{(1/q-1/p)}N^{2(1/p - 1/q)}\inorm{U}_{\mathbb L^p(-1,1)}
\end{equation}
for $1\le p\le q\le\infty$, where $\inorm{U}^q_{\mathbb L^q(-1,1)} = \int_{-1}^{1} |U|^q\,\mathrm{d}x$. Therefore, by equivalence of the norms, choosing $p=2$ and $q = \infty$ bounds the maximum norm by the discrete 2-norm
\begin{equation}
\inorm{U}_\infty \le C \Delta x^{-1/2}N \inorm{U}_2  \le C \Delta x^{-1/2}N \inorm{U}_N,
\label{eq:InfinityNormBound}
\end{equation}
where we have used the simpler alternate subscript notation for the norms.

Importantly, the LGL quadrature has the SBP property \cite{gassner2010}. For two polynomials $U$ and $W$,
\begin{equation}
\iprodN{U,W'} = \left\{ U_N W_N - U_0 W_0\right\} - \iprodN{U',W},
\end{equation}
where the prime denotes the derivative with respect to the argument. In particular, if $W = U$,
\begin{equation}
\iprodN{U,U'} = \oneHalf\left\{ U^2_N - U^2_0 \right\}.
\label{eq:UUprimInt}
\end{equation}
\subsection{Energy Bounds for the Scalar Problem}
To study stability of the scalar approximation, we start with $\phi = U$ in \eqref{eq:UDGSEMWeakForm} and use \eqref{eq:UUprimInt}, to find the bound on the energy of $U$
\begin{equation}
\frac{\Delta x}{2}\oneHalf\frac{d}{dt}\inorm{U}^2_N+U F^* |_{-1}^1 - \oneHalf \alpha U^2|_{-1}^1 = 0,
\end{equation}
or
\begin{equation}
\frac{\Delta x}{2}\oneHalf\frac{d}{dt}\inorm{U}^2_N+ U \left(F^* -\oneHalf F\right)|_{-1}^1 = 0,
\end{equation}
where $F = \alpha U$.
Summing over all elements gives
\begin{equation}
\oneHalf\frac{d}{dt} \sum_{k=1}^{K_u} \inorm{U}^2_{N,\underline{e}^k} + U \left.\left(F^* -\oneHalf F\right)\right|_{a}^c - \sum_{k=1}^{K_u-1}\left\{ \jump{U}F^* - \oneHalf \jump{UF}\right\}= 0,
\end{equation}
where $\jump{\cdot}$ is the jump between the left and right states. The second summation represents the contributions at the interior element interfaces.

The next steps, where the element interface and boundary inequalities are found, are standard. Here we consider the upwind numerical flux \eqref{eq:UpwindNumericalFlux}, which now can be written as
\begin{equation}
F^*\left(U^L,U^R\right) = \avg{F} - \oneHalf|\alpha|\jump{U},
\end{equation}
where $\avg{\cdot}$ is the average of the two states. Using the identity
\begin{equation}
\jump{pq} = \avg{p}\jump{q} + \jump{p}\avg{q},
\end{equation}
and the facts that $F = \alpha U$ and $\alpha$ is constant,
\begin{equation}
\jump{UF}= \avg{U}\jump{F} + \jump{U}\avg{F} = 2 \jump{U}\avg{F}.
\end{equation}
So,
\begin{equation}
\jump{U}F^* - \oneHalf \jump{UF} = \jump{U} \avg{F} - \oneHalf|\alpha|\jump{U}^2 - \jump{U}\avg{F} =  - \oneHalf|\alpha|\jump{U}^2
\end{equation}
adds dissipation,
and therefore
\begin{equation}
\oneHalf\frac{d}{dt} \sum_{k=1}^{K_u} \inorm{U}^2_{N,\underline{e}^k} + U \left.\left(F^* -\oneHalf F\right)\right|_{a}^c\le 0.
\label{eq:UEbound1}
\end{equation}

So now we find the influence of the subdomain boundaries. At the right, since $\alpha > 0$,
\begin{equation}
 U\left(F^* -\oneHalf F\right) = U\left( \frac{F^L + F^R}{2} - \frac{F^R - F^L}{2} - \oneHalf F^L \right) = \oneHalf \alpha U^2 \ge 0.
\end{equation}
At the left boundary, $u(a,t) = g(t) = 0$. Again with $\alpha > 0$, and completing the square
\begin{equation}
\begin{split}
 U\left(F^* -\oneHalf F\right) &= U\left( \frac{\alpha g + \alpha U}{2} - \frac{\alpha U - \alpha g}{2} - \oneHalf \alpha U \right) = -\oneHalf\alpha\left( U^2 - 2Ug\right)
 \\&
 = -\oneHalf\alpha\left( U^2 - 2Ug + g^2 - g^2\right) = -\oneHalf\alpha\left( (U-g)^2 - g^2\right)
 \\&
 = \oneHalf\alpha g^2 -\oneHalf\alpha (U-g)^2 .
 \end{split}
\end{equation}
Therefore,
\begin{equation}
\oneHalf\frac{d}{dt} \sum_{k=1}^{K_u} \inorm{U}^2_{N,\underline{e}^k}= -\mathcal D(U) + \oneHalf\alpha g^2,
\end{equation}
where $\mathcal D(U) =\oneHalf|\alpha|\jump{U}^2+ \oneHalf\alpha(U-g)^2 \ge 0$ is the upwind dissipation term.
Since $g=0$,
\begin{equation}
\sum_{k=1}^{K_u} \inorm{U}^2_{N,\underline{e}^k} \le \sum_{k=1}^{K_u} \inorm{\omega_0}^2_{N,\underline{e}^k}
\end{equation}
for any $K_u$ and $N$, so the approximation on the base grid is stable. If we define
\begin{equation}
\inorm{U}^2_{\Omega_u,N} \equiv \sum_{k=1}^{K_u} \inorm{U}^2_{N,\underline{e}^k},
\end{equation}
then
\begin{equation}
\inorm{U}_{\Omega_u,N}  \le \inorm{\omega_0}_{\Omega_u,N}.
\label{eq:ScalarUBound}
\end{equation}
Also, by the norm equivalence, \eqref{eq:NormEquivalence},
\begin{equation}
\inorm{U}_{\Omega_u}  \le \sqrt{3}\inorm{\omega_0}_{\Omega_u},
\label{eq:STDBaseBound}
\end{equation}
so the (continuous) norm of the approximate solution is bounded by the (initial) data. The bound \eqref{eq:STDBaseBound} is the standard one for the DGSEM.

The solution on the overset grid, $\bar e^k = [x_{k-1},x_k],\;  k=1,\ldots,K_v$, with $x_k = b + k\Delta x_k$, satisfies the same approximation except that its boundary condition at the leftmost element is the upwind value $g = U_*(b,t)$, where  $U_*(b,t)$ is the value of the polynomial in the element that contains $x=b$. As before, we assume constant $h=\Delta x_k$. Therefore,
\begin{equation}
\oneHalf\frac{d}{dt} \sum_{k=1}^K \inorm{V}^2_{N,\bar{e}^k} = \oneHalf\frac{d}{dt} \inorm{V}^2_{\Omega_v,N}= -\mathcal D(V) + \oneHalf\alpha U^2_*(b,t).
\end{equation}
Integrating in time,
\begin{equation}
\inorm{V}^2_{\Omega_v,N} \le \inorm{\omega_0}^2_{\Omega_v,N} + \alpha\int_0^T  U^2_*(b,t)dt.
\end{equation}
Stability is then determined by whether or not $U^2_*(b,t)$ can be bounded by data independent of $N$ and $K$ 
\cite{kreiss1970initial,gustafsson1995time,Nordstrom:2017yu}.

Since summation-by-parts applies only to an entire element, the discrete equivalent of Technique 1 is not available. To get a bound on the inflow value from the base mesh, however, we note that $U_*(b,t) \le \inorm{U}_\infty = \max_{\underline{e}^k} \inorm{U}_\infty$. Then by \eqref{eq:InfinityNormBound} and \eqref{eq:ScalarUBound},
\begin{equation}
U^2_*(b,t) \le C \Delta x^{-1}N^2\inorm{U}^2_{\Omega_u,N} \le  C \Delta x^{-1}N^2\inorm{\omega_0}^2_{\Omega_u,N}.
\end{equation}
\begin{rem}
We can bound the value $U(b,t)$ independent of time only because there is no coupling with, or feedback from, the overset domain and the trivial external boundary conditions. This will not be the case more generally where $\inorm{\statevec U}^2_{\Omega_u,N}$ will depend on $\statevec V$ due to feedback from the overset grid.
\end{rem}
Using the norm equivalence again, we have
\begin{lem}
The DGSEM approximations to the solutions of the overset grid problem for the first order scalar wave equation satisfy the bounds
\begin{equation}
\begin{gathered}
\inorm{U}^2_{\Omega_u}  \le \sqrt{3}\inorm{\omega_0}^2_{\Omega_u}\hfill\\
\inorm{V}^2_{\Omega_v} \le \sqrt{3}\inorm{\omega_0}^2_{\Omega_v} + C \Delta x^{-1}N^2  T \inorm{\omega_0}^2_{\Omega_u} .
\end{gathered}
\label{eq:ScalarBounds}
\end{equation}
\label{lem:EnergyBounds}
\end{lem}

\begin{rem}
Comparing \eqref{eq:ScalarBounds} with the analytical bounds \eqref{eq:Sharp1DScalarEstimate}, the approximation allows for possible growth in time and growth with $N$ and $\Delta x^{-1}$. We get the growth in time because the approximation does not have the discrete equivalent to Technique 1 \eqref{eq:ubarvalue}. Instead, we can only consider energy introduced to the overset grid from the entire base grid $\Omega_{u,N}$, or at least the entire element in which the overlap point lies.  Note that the high order approximation has a domain of dependence larger than the analytical one, and so downwind perturbations influence the solutions upwind.
\end{rem}

The question, now, is whether or not the interpolation condition for the scalar equation is destabilizing so that linear growth in time in \eqref{eq:ScalarBounds} is real and not just an artifact of weak bounds. To address that question, we examine the eigenvalues of the discrete system in \ref{AppendixB}. There, we show that using the interpolant of the base solution as the boundary value of the overset domain has no effect on the eigenvalues of the system matrix. Furthermore, in the absence of dissipation, the eigenvalues lie on the imaginary axis, so for the scalar problem the growth in time shown in \eqref{eq:ScalarBounds} is pessimistic.


We also have
\begin{thm}
For fixed polynomial order, $N$, mesh spacing $\Delta x$, and time, $T$, the total energy of the overset grid problem is bounded by the initial data.
\label{thm:ScalarBoundedEnergy}
\end{thm}
\begin{proof}
For fixed $N$ and $\Delta x$
\begin{equation}
\begin{split}
\inorm{V(T)}^2_{\Omega_v} &\le \sqrt{3}\inorm{\omega_0}^2_{\Omega_v} +  CT \inorm{\omega_0}^2_{\Omega_u} \\&
\le Ce^{\mu T}\left(\inorm{\omega_0}^2_{\Omega_u} + \inorm{\omega_0}^2_{\Omega_v}\right) \\&
\le 2Ce^{\mu T}\inorm{\omega_0}^2_{\Omega} \\&
= Ce^{\mu T}
\end{split}
\end{equation}
for some $\mu$ and some generic $C$ that depends on $N$ and $\Delta x$.
This bound is not sharp and it double counts the initial energy in the overlap domain. It also does not imply stability, since $C$ grows with $N\rightarrow \infty$ and/or $\Delta x\rightarrow 0$.
\end{proof}

\subsection{Energy Bounds for the Characteristic Approximation of a System of Equations}\label{sec:CharacteristicScheme}

Formally, we could use symmetrization and the results of the previous section as components to the system and combine them to get a final stability results. However that analysis breaks down when the coefficients are variable or if there is low order coupling. It also does not extend to multiple space dimensions. The more general approach is again to apply an energy method.

To study stability by the energy method, we set $\statevecGreek\phi \leftarrow \statevec U$ in \eqref{eqUDGSEMWeakForm}.
Using SBP and symmetry of the coefficient matrix,
\begin{equation}
\iprod{\mmatrix A \statevec U_\xi,\statevec U}_N = \left.\statevec U^T\statevec F\right|_{-1}^{1} - \iprodN{\statevec U,\mmatrix A\statevec U_\xi},
\end{equation}
so
\begin{equation}
\iprod{\mmatrix A \statevec U_\xi,\statevec U}_N = \oneHalf  \left.\statevec U^T\statevec F\right|_{-1}^{1}.
\label{eqSBPOnVolume}
\end{equation}
Therefore,
\begin{equation}
\frac{\Delta x}{2} \iprod{\statevec U_t,\statevec U}_N  + \left.\statevec U^T\left\{ \statevec F^*- \oneHalf \statevec F \right\} \right|_{-1}^1= 0,
\end{equation}
i.e.
\begin{equation}
\frac{\Delta x}{2} \oneHalf\frac{d}{dt}  \inormN{\statevec U}^2  +\left.\statevec U^T\left\{ \statevec F^*- \oneHalf \statevec F \right\} \right|_{-1}^1= 0.
\end{equation}
Summing over all elements gives
\begin{equation}
\oneHalf\frac{d}{dt} \sum_{k=1}^{K_u} \inorm{\statevec U}^2_{N,\underline{e}^k} + \statevec U^T \left.\left(\statevec F^* -\oneHalf \statevec F\right)\right|_{a}^c = \sum_{k=1}^{{K_u}-1}\left\{ \jump{\statevec U}^T\statevec F^* - \oneHalf \jump{\statevec U^T\statevec F}\right\}.
\label{eq:UDomainEnergy}
\end{equation}

The internal interfaces again introduce dissipation. As before, because $\mmatrix A$ is symmetric and constant,
\begin{equation}
\jump{\statevec U^T \statevec F}= \avg{\statevec U}^T\jump{\statevec F} + \jump{\statevec U}^T\avg{\statevec F} = 2 \jump{\statevec U}^T\avg{\statevec F}.
\end{equation}
Then, again,
\begin{equation}
\begin{split}
\left\{\jump{\statevec U}^T\statevec F^* - \oneHalf \jump{\statevec U^T\statevec F}\right\} &= \jump{\statevec U}^T\avg{\statevec F} - \oneHalf\jump{\statevec U}^T |\mmatrix A|\jump{\statevec U}- \jump{\statevec U}^T\avg{\statevec F}
\\ &= - \oneHalf\jump{\statevec U}^T |\mmatrix A|\jump{\statevec U}\le 0
\end{split}
\label{eq:RiemannInterfaceBound}
\end{equation}
introduces dissipation.

At the subdomain boundaries, we have
\begin{equation}
\begin{gathered}
BTL \equiv \statevec U^T \left.\left(\statevec F^*(\statevec g_L,\statevec U) -\oneHalf \statevec F\right)\right|_{a} \\
BTR \equiv \statevec U^T \left.\left(\statevec F^*(\statevec U, \statevec g_R) -\oneHalf \statevec F\right)\right|_{c},
\end{gathered}
\end{equation}
where $\statevec g_L$, $\statevec g_R$ are the left and right external states.

For $BTL$,
\begin{equation}
\begin{split}
BTL &= \statevec U^T\left(\mmatrix A^+\statevec g_L  + \mmatrix A^-\statevec U - \oneHalf (\mmatrix A^+ + \mmatrix A^- )\statevec U\right)
\\& = \statevec U^T\mmatrix A^+\statevec g_L - \oneHalf  \statevec U^T\mmatrix A^+  \statevec U - \oneHalf \statevec U^T|\mmatrix A^-|   \statevec U.
\end{split}
\end{equation}
We complete the square on the first two terms so
\begin{equation}
BTL = \oneHalf\statevec g_L^T\mmatrix A^+\statevec g_L -\oneHalf\left( \statevec g_L^T\mmatrix A^+\statevec g_L -2  \statevec U^T\mmatrix A^+\statevec g_L + \statevec U^T\mmatrix A^+  \statevec U\right) - \oneHalf \statevec U^T|\mmatrix A^-|   \statevec U.
\end{equation}
For convenience, let $\bar {\statevec g}_L \equiv \sqrt{\mmatrix A^+}\statevec g_L$, $\bar{\statevec U}_L \equiv \sqrt{\mmatrix A^+}\statevec U(a)$. Then
\begin{equation}
BTL = \oneHalf\bar {\statevec g}_L^T\bar {\statevec g}_L -\oneHalf\left(\bar{\statevec g}_L - \bar{\statevec U}_L\right)^2 - \left.\oneHalf \statevec U^T|\mmatrix A^-|   \statevec U\right|_a.
\label{eq:BTLSystem}
\end{equation}
The last two terms are non-positive (dissipative). The first depends only on the data.

Similarly,
\begin{equation}
BTR = -\oneHalf\bar {\statevec g}^T_R\bar {\statevec g}_R +\oneHalf\left(\bar{\statevec g}_R - \bar{\statevec U}_R\right)^2 + \left.\oneHalf \statevec U^T\mmatrix A^+   \statevec U\right|_c,
\end{equation}
where now $\bar {\statevec g}_R \equiv \sqrt{|\mmatrix A^-|}\statevec g$, $\bar{\statevec U}_R \equiv \sqrt{|\mmatrix A^-|}\statevec U(c)$.

Therefore, \eqref{eq:UDomainEnergy} becomes
\begin{equation}
\frac{d}{dt} \sum_{k=1}^{K_u} \inorm{\statevec U}^2_{N,\underline{e}^k} =   \statevec g_L^T\mmatrix A^+\statevec g_L  +\statevec g_R^T|\mmatrix A^-|\statevec g_R - \mathcal D,
\end{equation}
where $\mathcal D \ge 0$ is the dissipation term
\begin{equation}
\mathcal D = \sum_{k=1}^{{K_u}-1}\left\{\jump{\statevec U}^T |\mmatrix A|\jump{\statevec U}\right\} + \left(\bar{\statevec g}_L - \bar{\statevec U}_L\right)^2 + \left. \statevec U^T|\mmatrix A^-|   \statevec U\right|_a + \left(\bar{\statevec g}_R - \bar{\statevec U}_R\right)^2 + \left. \statevec U^T\mmatrix A^+   \statevec U\right|_c.
\label{eq:CharDissipation}
\end{equation}

Since we are taking $\statevec g_L  = 0$ and $\statevec g_R = \statevec V(c)$ for the base domain,
\begin{equation}
\frac{d}{dt}\inorm{\statevec U}^2_{\Omega_u,N} \equiv \frac{d}{dt} \sum_{k=1}^{K_u} \inorm{\statevec U}^2_{N,\underline{e}^k}  = \statevec  V^T(c)|\mmatrix A^-|\statevec  V(c) -\mathcal D\le \vnorm{\statevec V(c)}^2\inorm{|\mmatrix A^-|}_2 = \rho(|\mmatrix A^-|)\vnorm{\statevec V(c)}^2,
\label{eq:UBoundBase}
\end{equation}
where $\rho$ is the spectral radius and $\vnorm{\statevec V}$ is the Euclidean vector norm.

The solution on the overset grid satisfies the same weak form, but with $\statevec g_L = \statevec U(b)$ and $\statevec g_R = 0$, so
\begin{equation}
\frac{d}{dt}\inorm{\statevec V}^2_{\Omega_u,N} =\frac{d}{dt} \sum_{k=1}^{K_v} \inorm{\statevec V}^2_{N,\bar{e}^k} \le   \statevec U^T(b)\mmatrix A^+\statevec U(b) \le \vnorm{\statevec U(b)}^2\inorm{\mmatrix A^+}_2 =\rho(\mmatrix A^+)\vnorm{\statevec U(b)}^2.
\label{eq:VUBoundOverset}
\end{equation}

Equations \eqref{eq:UBoundBase} and \eqref{eq:VUBoundOverset} are coupled, so we add them together to get the time derivative of the {\it combined energy norm}
\begin{equation}
\frac{d}{dt}\inorm{\statevec U}^2_{\Omega_u,N} + \frac{d}{dt} \inorm{\statevec V}^2_{\Omega_v,N} \le \rho(\mmatrix A^+)\vnorm{\statevec U(b)}^2+\rho(|\mmatrix A^-|)\vnorm{\statevec V(c)}^2.
\label{eq:CombinedSystemBound0}
\end{equation}

We now need to bound the right hand side of \eqref{eq:CombinedSystemBound0}. As for the scalar equation, we can bound the infinity norm with the 2-norm of the approximation. For each component in $\statevec U$,
\begin{equation}
U_i(b) \le \max_x \;U_i \le C\Delta x^{-\oneHalf}N\inorm{U_i}_{\Omega_u,N}
\end{equation}
so
\begin{equation}
\vnorm{\statevec U(b)}^2 \le C \Delta x^{-1}N^2\inorm{\statevec U}_{\Omega_u,N}^2,
\end{equation}
and similarly for $\statevec V(c)$.
So we can bound the time derivative of the sum in \eqref{eq:CombinedSystemBound0} as
\begin{equation}
\begin{split}
\frac{d}{dt}\left\{\inorm{\statevec U}^2_{\Omega_u,N} + \inorm{\statevec V}^2_{\Omega_v,N}\right\} &\le C\rho(\mmatrix A^+)\Delta x^{-1}N^2\inorm{\statevec U}_{\Omega_u,N}^2+ C\rho(|\mmatrix A^-|)\Delta x^{-1}N^2\inorm{\statevec V}_{\Omega_v,N}^2
\\&
\le C\Delta x^{-1}N^2\rho(\mmatrix A)\left\{\inorm{\statevec U}_{\Omega_u,N}^2 +\inorm{\statevec V}_{\Omega_v,N}^2\right\},
\end{split}
\label{eq:CombinedBounddt}
\end{equation}
using the fact that $\rho(\mmatrix A^\pm)\le \rho(\mmatrix A)$. When we integrate \eqref{eq:CombinedBounddt} in time,
\begin{equation}
\inorm{\statevec U(T)}^2_{\Omega_u,N} + \inorm{\statevec V(T)}^2_{\Omega_v,N} \le \left\{\inorm{\statevecGreek \omega_0}^2_{\Omega_u,N} + \inorm{\statevecGreek \omega_0}^2_{\Omega_v,N}\right\} e^{C\Delta x^{-1}N^2 T\rho\left(\mmatrix A\right)},
\label{eq:FinalCharacteristicBound}
\end{equation}
we have shown that
\begin{thm}
The energy of the system problem with characteristic boundary conditions in one space dimension is bounded by the initial data for fixed polynomial order, mesh size, and time.
\end{thm}

\begin{rem}
Again we see that the bounds do not imply that the approximation is stable, but for fixed order and mesh size the approximate solutions can grow at most exponentially fast. The original problem does not have solutions that grow.
\end{rem}

From \eqref{eq:CharDissipation}, we see that dissipation in the characteristic formulation of the overset grid problem comes from three sources. The first is the inter-element dissipation from the upwind numerical flux. The second comes from the upwind numerical flux at the physical boundary points.  The final source is the dissipation associated with the upwind numerical flux at the overlap interface points at $x = b$ and $x = c$.

The question again arises as to whether or not the growth in the solutions implied by \eqref{eq:FinalCharacteristicBound} is due to lack of better bounds or is intrinsic to the coupling between the domains. In \ref{AppendixA}, we derive bounds for the dissipation-free approximation and show that the solutions remain bounded in the same way as in \eqref{eq:FinalCharacteristicBound}, but that the growth rate can be larger. In \ref{AppendixB} we show that the coupling terms are destabilizing in that the approximation allows for eigenvalues that lead to growth. This differs from the scalar case, and shows the pitfalls of extrapolating scalar results. The results of \ref{AppendixB} imply that dissipation is necessary for a solution to not blow up in time, and that the upwind numerical flux is sufficient to ensure the eigenvalues do not have positive real parts, at least for the one-dimensional problem. This suggests that, in one space dimension, the dissipation terms $\mathcal D$ in \eqref{eq:UBoundBase} (and implicit in \eqref{eq:VUBoundOverset}) are sufficient to counteract the growth term in \eqref{eq:FinalCharacteristicBound} for a fixed mesh and polynomial order. However, in general it is not known how much dissipation is necessary, a priori, and if it is sufficient to cancel the growth term in \eqref{eq:FinalCharacteristicBound} as $\Delta x\rightarrow 0,N\rightarrow\infty$.

\subsection{Energy Bounds for the Penalty Formulation}

To see how the energy evolves in the penalty formulation, we replace $\statevecGreek \phi\leftarrow \statevec U$ in \eqref{eq:DGSEMChimeraU},  and $\statevecGreek \varphi\leftarrow \statevec V$ in \eqref{eq:DGSEMChimeraV}, and use \eqref{eqSBPOnVolume} to obtain
\begin{equation}
 \iprod{\mathcal J\statevec U,\statevec U_t}_N
 + \left.\statevec U^T\left\{\hat{\statevec F}-\oneHalf\statevec F\right\}\right|^1_{-1}
 +
 \mathcal I^k_u(b)  \left.\statevec U^T\ \mmatrix{$\Sigma$}_u^b (\statevec U-\statevec V)\right|_b+
 \mathcal I^k_u(c)  \left.\statevec U^T\ \mmatrix{$\Sigma$}_u^c (\statevec U-\statevec V)\right|_c= 0
\end{equation}
and
\begin{equation}
\iprod{\mathcal J\statevec V,\statevec V_t}_{N }
+\left.\statevec V^T\left\{\hat{\statevec F}-\oneHalf\statevec F\right\}\right|^1_{-1}
+  \mathcal I^k_v(b)  \left.\statevec V^T \mmatrix{$\Sigma$}_v^b (\statevec V-\statevec U)\right|_b
+   \mathcal I^k_v(c)  \left.\statevec V^T \mmatrix{$\Sigma$}_v^c (\statevec V-\statevec U)\right|_c= 0.
\end{equation}
Summing over all the elements, and
taking into account the dissipativity of the numerical flux, \eqref{eq:RiemannInterfaceBound}, and the fact that $BLT|_a = 0$ on the base grid,
\begin{equation}
\oneHalf \frac{d}{dt}\inorm{\statevec U}^2_{\Omega_u,N}
+ \oneHalf \statevec U^T(c)\mmatrix A\statevec U(c)
+ \statevec U^T(b)\mmatrix{$\Sigma$}_u^b (\statevec U(b)-\statevec V(b))
+ \statevec U^T(c)\mmatrix{$\Sigma$}_u^c (\statevec U(c)-\statevec V(c))\le 0.
\label{eq:Energy1}
\end{equation}
Similarly,
\begin{equation}
\oneHalf \frac{d}{dt}\inorm{\statevec V}^2_{\Omega_v,N}
- \oneHalf \statevec V^T(b)\mmatrix A\statevec V(b)
+ \statevec V^T(b)\mmatrix{$\Sigma$}_v^b (\statevec V(b)-\statevec U(b))
+ \statevec V^T(c)\mmatrix{$\Sigma$}_v^c (\statevec V(c)-\statevec U(c))\le 0.
\label{eq:Energy2}
\end{equation}
Adding the two, the time derivative of the  combined energy norm is
\begin{equation}
\begin{split}
\oneHalf \frac{d}{dt}&\left\{ \inorm{\statevec U}^2_{\Omega_u,N} + \frac{d}{dt}\inorm{\statevec V}^2_{\Omega_v,N} \right\}
\\&
- \oneHalf \statevec V^T(b)\mmatrix A\statevec V(b)+ \statevec V^T(b)\mmatrix{$\Sigma$}_v^b (\statevec V(b)-\statevec U(b))+ \statevec U^T(b)\mmatrix{$\Sigma$}_u^b (\statevec U(b)-\statevec V(b))
\\&
+\oneHalf \statevec U^T(c)\mmatrix A\statevec U(c) + \statevec U^T(c)\mmatrix{$\Sigma$}_u^c (\statevec U(c)-\statevec V(c)) + \statevec V^T(c)\mmatrix{$\Sigma$}_v^c (\statevec V(c)-\statevec U(c))
\\& \le 0,
\end{split}
\label{eq:CombinedEnergy}
\end{equation}
which we will write as
\begin{equation}
 \frac{d}{dt}\mathcal E^2 + \mathcal P_b + \mathcal P_c \le 0,
\label{eq:CombinedEnergyShort}
\end{equation}
where $\mathcal E^2 = \inorm{\statevec U}^2_{\Omega_u,N} + \inorm{\statevec V}^2_{\Omega_v,N}$ and the penalty contributions are
\begin{equation}
\mathcal P_b \equiv -  \statevec V^T(b)\mmatrix A\statevec V(b)+ 2\statevec V^T(b)\mmatrix{$\Sigma$}_v^b (\statevec V(b)-\statevec U(b))+ 2\statevec U^T(b)\mmatrix{$\Sigma$}_u^b (\statevec U(b)-\statevec V(b)) ,
\label{eq:OrigPb}
\end{equation}
\begin{equation}
\mathcal P_c \equiv  \statevec U^T(c)\mmatrix A\statevec U(c) + 2\statevec U^T(c)\mmatrix{$\Sigma$}_u^c (\statevec U(c)-\statevec V(c)) + 2\statevec V^T(c)\mmatrix{$\Sigma$}_v^c (\statevec V(c)-\statevec U(c)).
\label{eq:OrigPc}
\end{equation}
\begin{rem}
Remember that the energy, $\mathcal E$, double counts the energy in the overlap region, so it bounds but does not match the norm of the solution on $\Omega$. Using norm equivalence,
\begin{equation}
 \inorm{\statevec U}^2_{\Omega_u,N} + \inorm{\statevec V}^2_{\Omega_v,N}  \ge E(\statevec U,\statevec V) \equiv \inorm{\statevec U}_{\Omega_u}^2 + \inorm{\statevec V}_{\Omega_v}^2  - \left\{ \eta \inorm{\statevec U}_{\Omega_{O}}^2 +
 (1-\eta)\inorm{\statevec V}_{\Omega_{O}}^2 \right\}
\end{equation}
for any $\eta\in (0,1)$. Again, the quantity on the right is the {\bf overset domain norm} that is equivalent to the true energy norm on $\Omega$ when $\statevec U = \statevec V$.
\end{rem}

The quantities $\mathcal P_{b,c}$ in \eqref{eq:OrigPb}-\eqref{eq:OrigPc} differ from those that appear in the overset grid norm in \cite{KOPRIVA2022110732}. Each is missing two terms due to the inability to remove the double counting in the overlap region, $\Omega_O$. Therefore, let us add and subtract those missing terms to match the analytical expressions,
\begin{equation}
\begin{split}
\mathcal P_b &= \eta\statevec U^T(b)\mmatrix A\statevec U(b) -  \eta\statevec V^T(b)\mmatrix A\statevec V(b)+ 2\statevec V^T(b)\mmatrix{$\Sigma$}_v^b (\statevec V(b)-\statevec U(b))+ 2\statevec U^T(b)\mmatrix{$\Sigma$}_u^b (\statevec U(b)-\statevec V(b))
\\& - \eta\statevec U^T(b)\mmatrix A\statevec U(b) -(1-\eta)\statevec V^T(b)\mmatrix A\statevec V(b)
\end{split}
\end{equation}
\begin{equation}
\begin{split}
\mathcal P_c &=  (1-\eta)\statevec U^T(c)\mmatrix A\statevec U(c) - (1-\eta)\statevec V^T(c)\mmatrix A\statevec V(c) + 2\statevec U^T(c)\mmatrix{$\Sigma$}_u^c (\statevec U(c)-\statevec V(c)) + 2\statevec V^T(c)\mmatrix{$\Sigma$}_v^c (\statevec V(c)-\statevec U(c))
\\& +\eta \statevec U^T(c)\mmatrix A\statevec U(c)+ (1-\eta)\statevec V^T(c)\mmatrix A\statevec V(c).
\end{split}
\end{equation}

With the added terms, we can re-write $\mathcal P_{b,c}$ in the matrix-vector form
\begin{equation}
\mathcal P_{b,c} = \left. \left[\begin{array}{c}\statevec U \\\statevec V\end{array}\right]^T\mmatrix M_{b,c}  \left[\begin{array}{c}\statevec U \\\statevec V\end{array}\right] \right|_{b,c} + Q_{b,c},
\end{equation}
\begin{equation}
\mmatrix M_b =\left[\begin{array}{cc} \eta\mmatrix A +2\mmatrix{$\Sigma$}^b_u& -(\mmatrix{$\Sigma$}^b_u  + \mmatrix{$\Sigma$}^b_v)\\-(\mmatrix{$\Sigma$}^b_u+ \mmatrix{$\Sigma$}^b_v) & -\eta\mmatrix A + 2\mmatrix{$\Sigma$}^b_v \end{array}\right]
,\quad\mmatrix M_c =\left[\begin{array}{cc} (1-\eta)A + 2\mmatrix{$\Sigma$}^c_u& -(\mmatrix{$\Sigma$}^c_u  + \mmatrix{$\Sigma$}^c_v)\\-(\mmatrix{$\Sigma$}^c_u+ \mmatrix{$\Sigma$}^c_v) &  -(1-\eta)\mmatrix A +2\mmatrix{$\Sigma$}^c_v \end{array}\right]
\end{equation}
and
\begin{equation}
Q_b = - \eta\statevec U^T(b)\mmatrix A\statevec U(b) -(1-\eta)\statevec V^T(b)\mmatrix A\statevec V(b),\quad Q_c =\eta \statevec U^T(c)\mmatrix A\statevec U(c)+ (1-\eta)\statevec V^T(c)\mmatrix A\statevec V(c)
\end{equation}
contain the missing terms.

With a proper choice of the coefficient matrices, $\mmatrix M_{b,c} \ge 0$, all but the missing terms contained in $Q_{b,c}$ are dissipative.
From \cite{KOPRIVA2022110732}, the conditions for $\mmatrix M_{b,c} \ge 0$ is guaranteed if $\mmatrix{$\Sigma$}_v>0$, $\mmatrix{$\Sigma$}_u>0$ and
\eqref{eq:1DSigmaMatrixConditionsWi} holds.
Then choosing the penalty matrices with the conditions 
\eqref{eq:1DSigmaMatrixConditionsWi},
\begin{equation}
\begin{split}
 \frac{d}{dt} \left\{ \inorm{\statevec U}^2_{\Omega_u,N} + \frac{d}{dt}\inorm{\statevec V}^2_{\Omega_v,N} \right\} \le 
 & \ \eta \statevec U^T(b)\mmatrix A\statevec U(b)+ (1-\eta)\statevec V^T(b)\mmatrix A\statevec V(b)
 \\& - \eta\statevec U^T(c)\mmatrix A\statevec U(c) -(1-\eta)\statevec V^T(c)\mmatrix A\statevec V(c).
 \end{split}
\end{equation}
Again splitting the matrix $\mmatrix A = \mmatrix A^+ - |\mmatrix A^-|$, and bounding terms that are non-positive,
\begin{equation}
\begin{split}
 \frac{d}{dt} \left\{ \inorm{\statevec U}^2_{\Omega_u,N} + \frac{d}{dt}\inorm{\statevec V}^2_{\Omega_v,N} \right\}
 \le & \ \eta \statevec U^T(b)\mmatrix A^+\statevec U(b)+ (1-\eta)\statevec V^T(b)\mmatrix A^+\statevec V(b)
 \\&+\eta\statevec U^T(c)|\mmatrix A^-|\statevec U(c) +(1-\eta)\statevec V^T(c)|\mmatrix A^-|\statevec V(c).
 \end{split}
\end{equation}
Then as in \eqref{eq:VUBoundOverset},
\begin{equation}
\begin{split}
 \frac{d}{dt} \left\{ \inorm{\statevec U}^2_{\Omega_u,N} + \frac{d}{dt}\inorm{\statevec V}^2_{\Omega_v,N} \right\}
 \le
& \ \eta\rho(\mmatrix A^+)\vnorm{\statevec U(b)}^2+ (1-\eta)\rho(\mmatrix A^+)\vnorm{\statevec V(b)}^2
\\&+ \eta\rho(|\mmatrix A^-|)|{\statevec U(c)}|^2 +(1-\eta)\rho(|\mmatrix A^-|)|{\statevec V(c)}|^2.
\end{split}
\end{equation}
But
\begin{equation}
\vnorm{\statevec U(b)}^2 \le C\Delta x^{-1}N^2\inorm{\statevec U}_{\Omega_u,N}^2,\quad \vnorm{\statevec V(c)}^2 \le C\Delta x^{-1}N^2\inorm{\statevec V}_{\Omega_v,N}^2,
\end{equation}
so
\begin{equation}
\begin{split}
 \frac{d}{dt} \left\{ \inorm{\statevec U}^2_{\Omega_u,N} + \frac{d}{dt}\inorm{\statevec V}^2_{\Omega_v,N} \right\}
 \le C\Delta x^{-1}N^2 &\left\{  \eta\rho(\mmatrix A^+)\inorm{\statevec U}_{\Omega_u,N}^2+  (1-\eta)\rho(\mmatrix A^+)\inorm{\statevec V}_{\Omega_v,N}^2 \right.
 \\&\left. +  \eta\rho(|\mmatrix A^-|)\inorm{\statevec U}_{\Omega_u,N}^2 + (1-\eta)\rho(|\mmatrix A^-|)\inorm{\statevec V}_{\Omega_v,N}^2 \right\}.
 \end{split}
\end{equation}
Since $\rho(\mmatrix A^\pm)\le \rho(A)$,
\begin{equation}
 \frac{d}{dt} \left\{ \inorm{\statevec U}^2_{\Omega_u,N} + \frac{d}{dt}\inorm{\statevec V}^2_{\Omega_v,N} \right\}
 \le \Delta x^{-1}N^2\rho(\mmatrix A)\left\{ \eta\inorm{\statevec U}^2_{\Omega_u,N} + (1-\eta)\inorm{\statevec V}^2_{\Omega_v,N} \right\}
\end{equation}
and since $\eta\in(0,1)$,
\begin{equation}
 \frac{d}{dt} \left\{ \inorm{\statevec U}^2_{\Omega_u,N} + \frac{d}{dt}\inorm{\statevec V}^2_{\Omega_v,N} \right\}
 \le C\Delta x^{-1}N^2\rho(\mmatrix A)\left\{ \inorm{\statevec U}^2_{\Omega_u,N} + \inorm{\statevec V}^2_{\Omega_v,N} \right\},
\end{equation}
which leads to the same bound as for the characteristic approximation \eqref{eq:FinalCharacteristicBound}
\begin{equation}
\inorm{\statevec U(T)}^2_{\Omega_u,N} + \inorm{\statevec V(T)}^2_{\Omega_v,N} \le \left\{\inorm{\statevecGreek \omega_0}^2_{\Omega_u,N} + \inorm{\statevecGreek \omega_0}^2_{\Omega_v,N}\right\} e^{C\Delta x^{-1}N^2 \rho\left(\mmatrix A \right)T},
\label{eq:FinalPenaltyBound}
\end{equation}
but with different dissipation terms.

We can derive sufficient conditions for $\mmatrix{$\Sigma$}_v$ that satisfy \eqref{eq:1DSigmaMatrixConditionsWi} using the splitting $\mmatrix A = \mmatrix A^+ -|\mmatrix A^-|$,
\begin{equation}
2 \mmatrix{$\Sigma$}_v - \beta\mmatrix A^+ + \beta|\mmatrix A^-|\ge 0,
\end{equation}
which is satisfied if $\mmatrix{$\Sigma$}_v\ge  \frac{1}{2} \mmatrix A^+ > \frac{\beta}{2} \mmatrix A^+$, since $0< \beta < 1$. Then by the second relation in \eqref{eq:1DSigmaMatrixConditionsWi},
\begin{equation}
\beta \mmatrix A^+ - \beta |\mmatrix A^-| + \mmatrix{$\Sigma$}_u = \mmatrix{$\Sigma$}_v \ge \frac{\beta}{2} \mmatrix A^+,
\end{equation}
so
\begin{equation}
\frac{\beta}{2} \mmatrix A^+ -|\mmatrix A^-| + \mmatrix{$\Sigma$}_u \ge 0,
\end{equation}
which is guaranteed if
\begin{equation}
 \mmatrix{$\Sigma$}_u\ge \max\left(|\mmatrix A^-| - \frac{\beta}{2} \mmatrix A^+,\mmatrix 0\right).
 \label{eq:SigmaUBounds}
 \end{equation}

Therefore we have proved:
\begin{thm}
The DGSEM Approximations \eqref{eq:DGSEMChimeraU} and \eqref{eq:DGSEMChimeraV} to the overset grid problem \eqref{eq:WeakFormsonFullDomains} satisfy the bound
\begin{equation}
\inorm{\statevec U(T)}^2_{\Omega_u,N} + \inorm{\statevec V(T)}^2_{\Omega_v,N} \le \left\{\inorm{\statevecGreek \omega_0}^2_{\Omega_u,N} + \inorm{\statevecGreek \omega_0}^2_{\Omega_v,N}\right\} e^{C\Delta x^{-1}N^2 T\rho\left(\mmatrix A\right)},
\label{eq:FinalChimeraCharacteristicBound}
\end{equation}
if
\begin{equation}
\mmatrix{$\Sigma$}_v\ge \frac{1}{2} \mmatrix A^+,\quad \mmatrix{$\Sigma$}_u\ge |\mmatrix A^-|.
\label{eq:SigmaDGSEMConditionsP}
\end{equation}
\end{thm}

For the penalty formulation, dissipation also comes from three sources. As before, there is inter-element and physical boundary point dissipation when the upwind numerical flux is used. Now, however, there is dissipation from the coupling terms at $x=b$ and $x=c$ that appear in both subdomains, \eqref{eq:WeakFormsonFullDomains}.
Analytically, those terms remove the growth term. Numerically they add dissipation to inhibit that growth. 

To illustrate the effects of the dissipation terms, we look at the eigenvalues (as derived in \ref{AppendixB}) for $N=5$ for one element each in the domains $\Omega_u = [0,2]$, $\Omega_v = [1.1,3.5]$, and three configurations:
\begin{enumerate}
\item The dissipation-free formulation using the central numerical flux, \ref{AppendixA}
\item The characteristic formulation, \eqref{eqUDGSEMWeakForm}, and
\item  The penalty formulation \eqref{eq:DGSEMChimeraU}-\eqref{eq:DGSEMChimeraV} with penalty matrices $\mmatrix{$\Sigma$}_v= \frac{1}{2} \mmatrix A^+,\quad \mmatrix{$\Sigma$}_u= |\mmatrix A^-|.
$
\end{enumerate}

Fig.~\ref{fig:Various.pdf} shows that the dissipation-free approximation has unstable eigenvalues, whereas the upwind numerical flux supplies enough dissipation to move all eigenvalues to the left half plane, as shown in \ref{AppendixB}. The additional dissipation due to the penalty terms moves the eigenvalues even more to the left. Note that the eigenvalues of the upwind scheme of Sec.~\ref{sec:CharacteristicScheme} are each {\it repeated}, so that it appears to have half as many on the plot. The penalty formulation breaks that repetition. Repeated eigenvalues introduce the possibility of growth in time.
\begin{figure}[htbp] 
   \centering
   \includegraphics[width=2.50in]{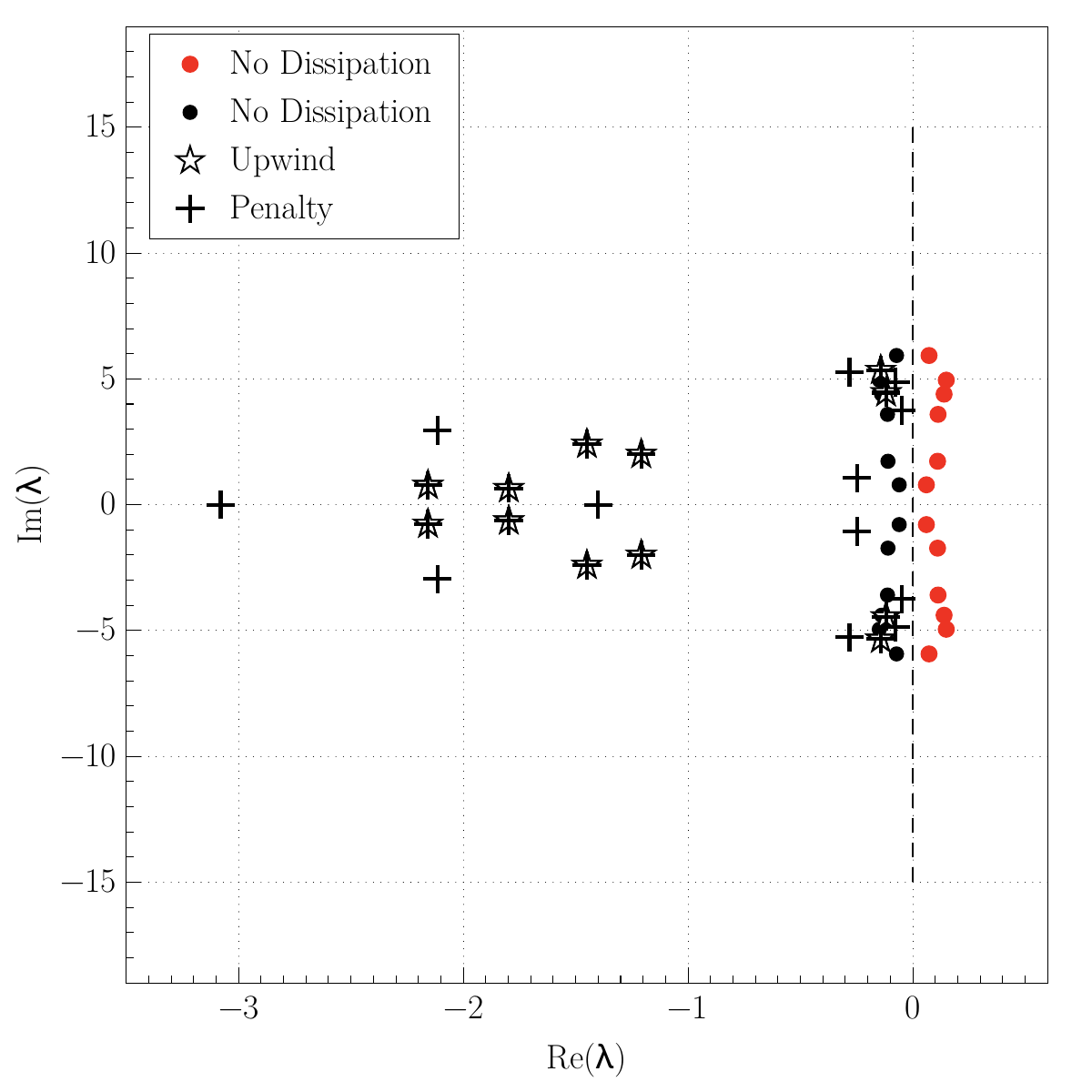}
   \includegraphics[width=2.5in]{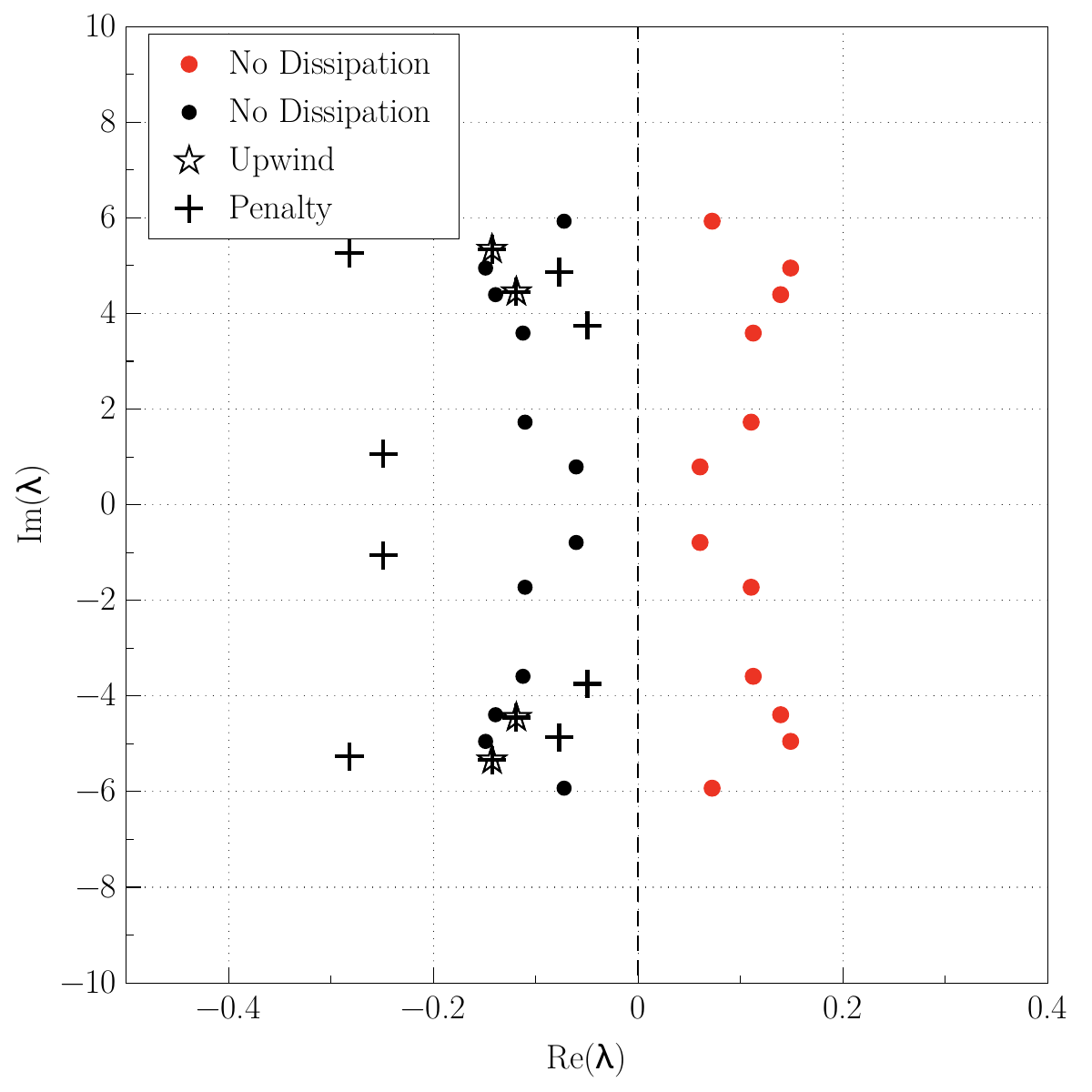}
   \caption{Eigenvalues for two elements for three approximations: No dissipation (\ref{AppendixA}), the upwind characteristic formulation \eqref{eqUDGSEMWeakForm}, and the penalty formulation \eqref{eq:DGSEMChimeraU}-\eqref{eq:DGSEMChimeraV}. The graph on the right enlarges the region near the imaginary axis. Unstable eigenvalues are drawn in red.}
   \label{fig:Various.pdf}
\end{figure}

\subsubsection{Enhancing Dissipation with Additional Coupling within the Overlap Region}

The penalty formulation in \cite{KOPRIVA2022110732} allows one to more tightly couple the solutions in the two domains by introducing penalty terms in the overlap region. In the weak form \eqref{eq:WeakFormsonFullDomains}, one adds penalties for an arbitrary number, $M$, of locations,
\begin{equation}
\begin{gathered}
 \iprod{\statevecGreek\phi,\statevec u_t}_{\Omega_u} + \iprod{\statevecGreek\phi,\mmatrix A u_x }_{\Omega_u}+
 \left.\statevecGreek\phi^T\ \mmatrix{$\Sigma$}_u^b (\statevec u-\statevec v)\right|_b+
 \left.\statevecGreek\phi^T\ \mmatrix{$\Sigma$}_u^c (\statevec u-\statevec v)\right|_c+\frac{1}{M}\sum_{m=1}^M\statevecGreek\phi^T\mmatrix{$\Sigma$}_u^m (\statevec u-\statevec v)= 0,\hfill\\
\iprod{\statevecGreek\varphi,\statevec v_t}_{\Omega_v} + \iprod{\statevecGreek\varphi,\mmatrix A \statevec v_x}_{\Omega_v} +  \left.\statevecGreek\varphi^T \mmatrix{$\Sigma$}_v^b (\statevec v-\statevec u)\right|_b+   \left.\statevecGreek\varphi^T \mmatrix{$\Sigma$}_v^c (\statevec v-\statevec u)\right|_c+\frac{1}{M}\sum_{m=1}^M\statevecGreek\varphi^T\mmatrix{$\Sigma$}_v^m (\statevec v-\statevec u)= 0.\hfill\\
\end{gathered}
\label{eq:WeakFormsFonFullDomainsWOvrlap}
\end{equation}
The overset domain problem remains well-posed and equivalent to the OP when the penalty matrices satisfy \cite{KOPRIVA2022110732}
\begin{equation}
(1-\eta)\mmatrix{$\Sigma$}^m_u = \eta\mmatrix{$\Sigma$}_v^m,
\label{eq:InteriorPenaltyCoupling}
\end{equation}
where, again, $\eta$ is the coupling parameter in the norm, \eqref{eq:EnergyNorm}. When added to the approximation \eqref{eq:DGSEMChimeraU}-\eqref{eq:DGSEMChimeraV}, the coupling terms add
\begin{equation}
\mathcal P = \frac{1}{M}\sum_{m=1}^M \mathcal P_m = \frac{1}{M}\sum_{m=1}^M \statevec U^T\mmatrix{$\Sigma$}^m_u(\statevec U - \statevec V) +\statevec V^T\mmatrix{$\Sigma$}^m_v(\statevec V - \statevec U)
\end{equation}
to the left side of the energy, \eqref{eq:CombinedEnergyShort}.
Each term in the sum contributes to the energy an amount of
\begin{equation}
\mathcal P_m = \statevec U^T\mmatrix{$\Sigma$}^m_u(\statevec U - \statevec V) +\statevec V^T\mmatrix{$\Sigma$}^m_v(\statevec V - \statevec U) =  \left[\begin{array}{c}\statevec U \\\statevec V\end{array}\right]^T\mmatrix M_m  \left[\begin{array}{c}\statevec U \\\statevec V\end{array}\right],
\end{equation}
where
\begin{equation}
\mmatrix M_m = \left[\begin{array}{cc} \mmatrix{$\Sigma$}^m_u& -\mmatrix{$\Sigma$}^m_u \\-\mmatrix{$\Sigma$}^m_v &\mmatrix{$\Sigma$}^m_v \end{array}\right]
\end{equation}
should be constructed so that $\mathcal P_m\ge 0$.

From \cite{KOPRIVA2022110732}, the conditions for $\mmatrix M_m \ge 0$ can be found by
rotating each matrix $\mmatrix M_m$ using
\begin{equation}
 \left[\begin{array}{c}\statevec U \\\statevec V\end{array}\right]
 = \oneHalf \left[\begin{array}{cc}I & I \\-I & I\end{array}\right] \left[\begin{array}{c}\statevec U - \statevec V \\\statevec U + \statevec V\end{array}\right] \equiv \mmatrix R \left[\begin{array}{c}\statevec U - \statevec V \\\statevec U + \statevec V\end{array}\right].
\end{equation}
Then
\begin{equation}
\mathcal P_m = \left[\begin{array}{c}\statevec U - \statevec V \\\statevec U +\statevec V\end{array}\right]^T\mmatrix R^T \mmatrix M_m \mmatrix R \left[\begin{array}{c}\statevec U - \statevec V \\\statevec U +\statevec V\end{array}\right] \equiv  \left[\begin{array}{c}\statevec U - \statevec V \\\statevec U +\statevec V\end{array}\right]^T\tilde{\mmatrix M}_m\left[\begin{array}{c}\statevec U - \statevec V \\\statevec U +\statevec V\end{array}\right].
\label{eq:MMatrixRotation}
\end{equation}
The rotated matrix,
\begin{equation}
\tilde{\mmatrix M}_m =\frac{1}{4} \left[\begin{array}{cc} 2(\mmatrix{$\Sigma$}^m_u+\mmatrix{$\Sigma$}^m_u)& 0 \\ 2(\mmatrix{$\Sigma$}^m_u-\mmatrix{$\Sigma$}^m_v) &0 \end{array}\right]
\end{equation}
is positive semi-definite if $\mmatrix{$\Sigma$}^m_u = \mmatrix{$\Sigma$}^m_v \ge 0$. Setting $\mmatrix{$\Sigma$}^m = \mmatrix{$\Sigma$}^m_u = \mmatrix{$\Sigma$}^m_v$, the lower left corner of $\tilde{\mmatrix M}_m$ is zeroed and the time rate of change of the energy is decreased by an amount
\begin{equation}
\mathcal P_m = (\statevec U - \statevec V)^T\mmatrix{$\Sigma$}^m (\statevec U - \statevec V)\ge 0.
\label{eq:OPDissipation}
\end{equation}
In fact,
\begin{equation}
\mathcal P_m = \statevec U^T\mmatrix{$\Sigma$}^m_u(\statevec U - \statevec V) +\statevec V^T\mmatrix{$\Sigma$}^m_v(\statevec V - \statevec U)=   \left(\statevec U^T\mmatrix{$\Sigma$}^m_u-\statevec V^T\mmatrix{$\Sigma$}^m_v\right) (\statevec U - \statevec V).
\end{equation}
So if we start by assuming that $\mmatrix{$\Sigma$} = \mmatrix{$\Sigma$}_u = \mmatrix{$\Sigma$}_v$, then we immediately have \eqref{eq:OPDissipation}.

Although one has a lot of flexibility when designing $\mmatrix{$\Sigma$}^m$, a simple choice is a diagonal matrix with one parameter, $\mmatrix{$\Sigma$}^m = \varepsilon\mmatrix I$. With that choice,
\begin{equation}
\mathcal P_m = \varepsilon |\statevec U - \statevec V|^2\ge 0.
\end{equation}
\begin{rem}
The amount of the overlap penalty dissipation, therefore, depends only on the amount by which the solutions in the two subdomains differ, just like dissipation from the numerical flux, \eqref{eq:RiemannInterfaceBound}. This point is important, because it says that the overall dissipation of the scheme can be adjusted at an arbitrary number of points in a way that does not affect the formal order of accuracy of the approximation, as opposed, for example, to adding a standard even derivative artificial viscosity term to each equation.
\end{rem}


We observe the effect of the overlap penalty terms by comparing the eigenvalues (\ref{AppendixB}) with and without those terms. As before, we use one element each in the two subdomains $\Omega_u = [0,2]$, $\Omega_v = [1.1,3.5]$, and polynomial order $N=5$. No dissipation is added except for the overlap penalty, where we choose $M=2$  and use the same penalty matrix $\mmatrix{$\Sigma$} = \varepsilon \mmatrix I >0$ at each point. In Fig.~\ref{fig:OverlapPenaltyEVs} we compare the eigenvalues of the original dissipation-free approximation ($\varepsilon=0$) to those one gets when two overlap penalties are applied. The eigenvalues shift to the left as the penalty parameter increases, and for the two non-zero values presented, all are stable. However unstable eigenvalues are still present for $\varepsilon=0.5$.

\begin{figure}[htbp] 
   \centering
   \includegraphics[width=3.0in]{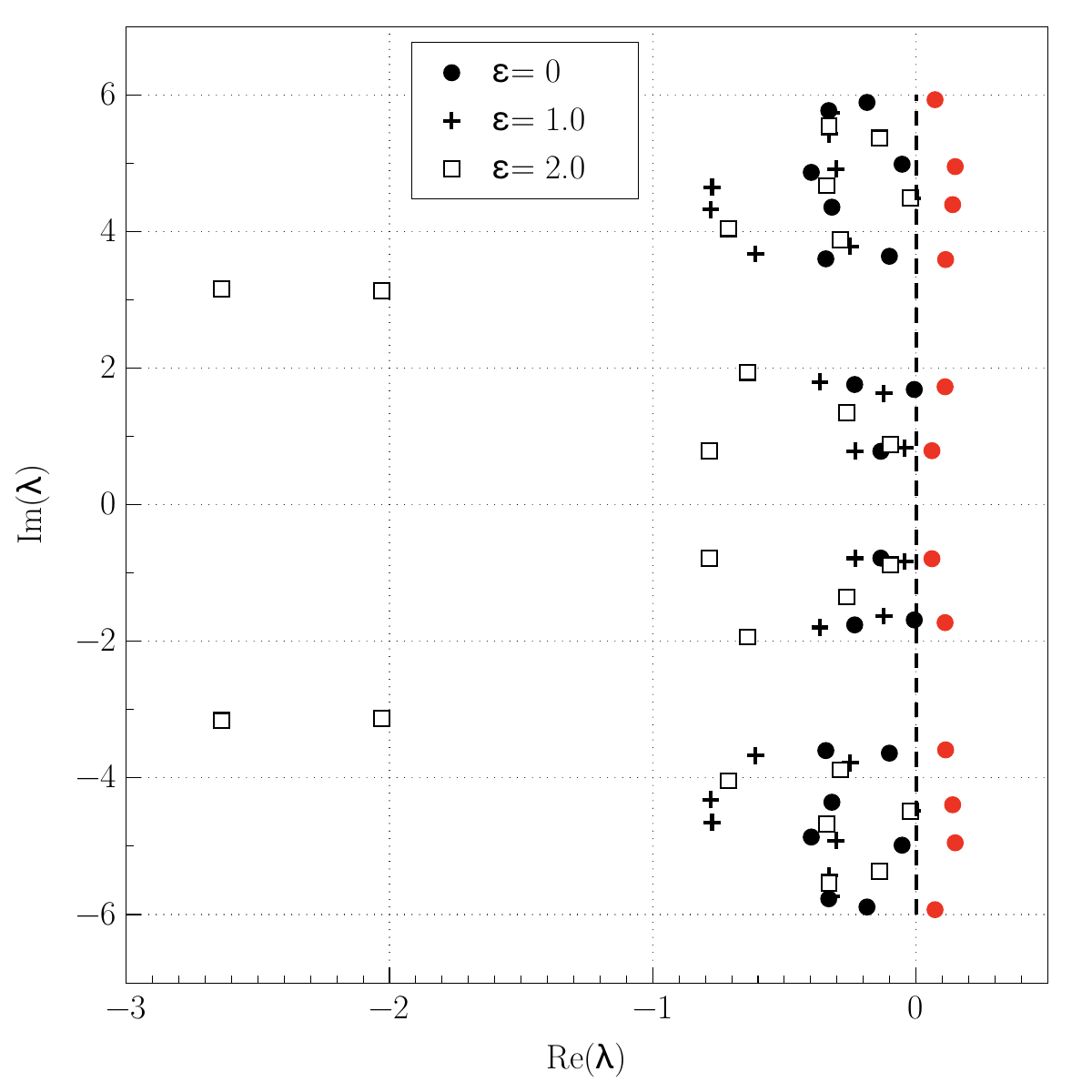}
   \caption{Eigenvalues for the penalty formulation \eqref{eq:DGSEMChimeraU}-\eqref{eq:DGSEMChimeraV} ($\varepsilon=0$) compared to those for the penalty formulation plus two overlap penalty points, \eqref{eq:WeakFormsFonFullDomainsWOvrlap}. Unstable eigenvalues are drawn in red.}
   \label{fig:OverlapPenaltyEVs}
\end{figure}

\section{Implementation of the Penalty Formulation}\label{sec:Implementation}

The DGSEM uses the numerical flux at the element interfaces, and can be easily modified to implement the penalty formulation. One form is to use \eqref{eq:DGSEMChimeraU}-\eqref{eq:DGSEMChimeraV}, where one replaces the numerical flux and adds the penalty terms. An algebraically equivalent form that is simpler to implement substitutes for the numerical flux $\hat {\statevec F}$ and converts the pure penalty formulation into the standard DGSEM plus new penalty terms,
\begin{equation}
\begin{split}
 \iprod{\mathcal J\statevecGreek\phi,\statevec U_t}_N
 &+ \left.\statevecGreek\phi^T\left\{{\statevec F^*}-\statevec F\right\}\right|^1_{-1} +\iprod{\statevecGreek\phi,\mmatrix A \statevec U_\xi }_{N }
 \\&+
 \mathcal I^k_u(b)  \left.\statevecGreek\phi^T\ \mmatrix{$\Sigma$}_u^b (\statevec U-\statevec V)\right|_b+
 \mathcal I^k_u(c)  \left.\statevecGreek\phi^T\left(\mmatrix{$\Sigma$}_u^c (\statevec U-\statevec V) +\statevec F(\statevec U) - \statevec F^*\right)\right|_c= 0
 \end{split}
\end{equation}
and
\begin{equation}
\begin{split}
\iprod{\mathcal J\statevecGreek\varphi,\statevec V_t}_{N }
&+\left.\statevecGreek\varphi^T\left\{{\statevec F^*}-\statevec F\right\}\right|^1_{-1}
+ \iprod{\statevecGreek\varphi,\mmatrix A \statevec V_\xi}_{N }
\\&+  \mathcal I^k_v(b)  \left.\statevecGreek\varphi^T \left( \mmatrix{$\Sigma$}_v^b (\statevec V-\statevec U) +\statevec F^* - \statevec F(\statevec V)\right)\right|_b
+   \mathcal I^k_v(c)  \left.\statevecGreek\varphi^T \mmatrix{$\Sigma$}_v^c (\statevec V-\statevec U)\right|_c= 0.
 \end{split}
\end{equation}

See \cite{Kopriva:2009nx},\cite{Winters2021} for how to convert these weak forms to a system of ODEs in time for the nodal values of the solutions. In this paper, the system of ODEs is approximately integrated in time using a third order Runge-Kutta method \cite{Williamson:1980:JCP80}, with time steps chosen so that the temporal error is small compared to the spatial error.

Some special choices for the penalty matrices can narrow the options and simplify the implementation further, as seen in the following examples:
\begin{ex}
If one chooses $\mmatrix{$\Sigma$}_{v} =  \mmatrix A^+$, $\mmatrix{$\Sigma$}_{u} =  |\mmatrix A^-|$, which satisfy the conditions \eqref{eq:SigmaDGSEMConditionsP}, and the upwind numerical flux, \eqref{eq:UpwindNumericalFlux}, for $\statevec F^*$, then
\begin{equation}
\begin{split}
\mmatrix{$\Sigma$}_u^c (\statevec U-\statevec V) +\statevec F(\statevec U) - \statevec F^* &= |\mmatrix A^-|(\statevec U-\statevec V) + \mmatrix A^+\statevec U + \mmatrix A^-\statevec U - \left( \mmatrix A^+ \statevec U + \mmatrix A^- \statevec V\right)
\\&
= 0.
\end{split}
\end{equation}
Similarly,
\begin{equation}
\begin{split}
\mmatrix{$\Sigma$}_v^b (\statevec V-\statevec U) +\statevec F^* - \statevec F(\statevec V) &=  \mmatrix A^+(\statevec V-\statevec U) + \mmatrix A^+\statevec U + \mmatrix A^-\statevec V -\mmatrix A^+\statevec V - A^-\statevec V
\\&
= 0.
\end{split}
\end{equation}
Then the DGSEM approximation of the penalty formulation \eqref{eq:WeakFormsonFullDomains} is equivalent to the standard DGSEM plus a dissipative penalty applied only to the overlap interface point, i.e.,
\begin{equation}
 \iprod{\mathcal J\statevecGreek\phi,\statevec U_t}_N
 + \left.\statevecGreek\phi^T\left\{{\statevec F^*}-\statevec F\right\}\right|^1_{-1} +\iprod{\statevecGreek\phi,\mmatrix A \statevec U_\xi }_{N }
+
 \mathcal I^k_u(b)  \left.\statevecGreek\phi^T |\mmatrix A^-| (\statevec U-\statevec V)\right|_b
= 0
\end{equation}
and
\begin{equation}
\begin{split}
\iprod{\mathcal J\statevecGreek\varphi,\statevec V_t}_{N }
&+\left.\statevecGreek\varphi^T\left\{{\statevec F^*}-\statevec F\right\}\right|^1_{-1}
+ \iprod{\statevecGreek\varphi,\mmatrix A \statevec V_\xi}_{N }
+   \mathcal I^k_v(c)  \left.\statevecGreek\varphi^T\mmatrix A^+ (\statevec V-\statevec U)\right|_c= 0.
 \end{split}
\end{equation}
\end{ex}
\begin{rem}
Using the final result of \ref{AppendixB}, the choice of penalty matrices in Example 1 is sufficient to stabilize the eigenvalues of the system in one space dimension.
\end{rem}
\begin{ex}
One can adjust the coupling and the amount of dissipation of Example 1 by adjusting the size of the penalty matrices. Let $\gamma_u\ge 1$, $\gamma_v\ge \oneHalf$ and let
\begin{equation}
 \mmatrix{$\Sigma$}_{u} =  \gamma_u |\mmatrix A^-|,\quad \mmatrix{$\Sigma$}_{v} =  \gamma_v \mmatrix A^+.
\end{equation}
Then the DGSEM of the penalty formulation is equivalent to
\begin{equation}
\begin{split}
 \iprod{\mathcal J\statevecGreek\phi,\statevec U_t}_N
 &+ \left.\statevecGreek\phi^T\left\{{\statevec F^*}-\statevec F\right\}\right|^1_{-1} +\iprod{\statevecGreek\phi,\mmatrix A \statevec U_\xi }_{N }
 \\&+
 \mathcal I^k_u(b)  \left.\gamma_u\statevecGreek\phi^T |\mmatrix A^-| (\statevec U-\statevec V)\right|_b+
 \mathcal I^k_u(c)  \left.(\gamma_u-1)\statevecGreek\phi^T |\mmatrix A^-|(\statevec U-\statevec V)\right|_c= 0
 \end{split}
 \label{eq:gammaVersionU}
\end{equation}
and
\begin{equation}
\begin{split}
\iprod{\mathcal J\statevecGreek\varphi,\statevec V_t}_{N }
&+\left.\statevecGreek\varphi^T\left\{{\statevec F^*}-\statevec F\right\}\right|^1_{-1}
+ \iprod{\statevecGreek\varphi,\mmatrix A \statevec V_\xi}_{N }
\\&+  \mathcal I^k_v(b)  \left. (\gamma_v-1) \statevecGreek\varphi^T\mmatrix A^+ (\statevec V-\statevec U)\right|_b
+   \mathcal I^k_v(c)  \left. \gamma_v \statevecGreek\varphi^T\mmatrix A^+ (\statevec V-\statevec U)\right|_c= 0.
 \end{split}
 \label{eq:gammaVersionV}
\end{equation}
\end{ex}

Modification of an existing DGSEM implementation requires two additions. The first is to find the reference space locations of the points $b$ and $c$ for interpolation from the donor domain by search and rootfinding. The second is to add the penalty terms in those elements where the penalty is active ($\mathcal I^k = 1$).

\section{Computed Examples}
We provide examples of the approximations \eqref{eq:gammaVersionU}-\eqref{eq:gammaVersionV} of Sec.~\ref{sec:Implementation} to the penalty formulation \eqref{eq:WeakFormsonFullDomains} of the symmetric hyperbolic system from Sec.~8.1.6 of \cite{Kopriva:2009nx},
\begin{equation}
\left[ {\begin{array}{*{20}c}
   \omega_1  \\
   \omega_2  \\
\end{array}} \right]_t  + \left[ {\begin{array}{*{20}c}
   0 & 1  \\
   1 & 0  \\
\end{array}} \right]\left[ {\begin{array}{*{20}c}
   \omega_1  \\
   \omega_2  \\
\end{array}} \right]_x  = \left[ {\begin{array}{*{20}c}
   \omega_1  \\
   \omega_2  \\
\end{array}} \right]_t  + \left[ {\begin{array}{*{20}c}
   \omega_2  \\
   \omega_1  \\
\end{array}} \right]_x  = 0\quad x \in \left[ {a,d} \right].
\label{eq:1DWaveEquationSystem}
\end{equation}
We approximate the system with a periodic solution in space and time, but with external states specified by the analytic solution, so that the resolution requirements are uniform over the domain. The choice of non-periodic boundary conditions means that there is dissipation and addition of energy at the left and right physical boundaries. Having the elements the same size and at the same polynomial order, any differences should be due to the choice of penalty matrices.

For the exact periodic solution we choose
\begin{equation}
\begin{gathered}
\omega_1 = \cos(k(x-t)) + \sin(k(x+t))\hfill\\
\omega_2 = \cos(k(x-t)) - \sin(k(x+t))\hfill\\
\end{gathered}
\label{eq:ExactPeriodicSolution}
\end{equation}
with $k=4$.
The subdomain boundaries are given in Table \ref{tab:DomainBoundaries1D},
where $o$ is an offset to ensure that the subdomain boundaries fall within an element. 
\begin{table}[!htp]
\caption{Subdomain Boundary Locations}
\begin{center}
\label{tab:DomainBoundaries1D}
\begin{tabular}{cccc}a & b & c & d \\$0 +o$ & 3 & 5 + $o$ & 8\end{tabular}
\end{center}
\end{table}
(With element boundaries matching, the overlap interfaces become essentially the standard DGSEM.)

 As an example, exact and computed solutions are shown in Fig.~\ref{fig:SinusoidSoln} for time $T=25.0$ for $\mmatrix{$\Sigma$}_v= \oneHalf \mmatrix A^+$ and $ \mmatrix{$\Sigma$}_u= |\mmatrix A^-|$, i.e. $\gamma_v = \oneHalf$, $\gamma_u = 1$, which minimally satisfy \eqref{eq:SigmaDGSEMConditionsP}, and $o=0.25$. Unless otherwise noted, all results will be presented for this final time. Six equally sized elements divide each of the two domains.
\begin{figure}[htbp] 
   \centering
   \includegraphics[width=4in]{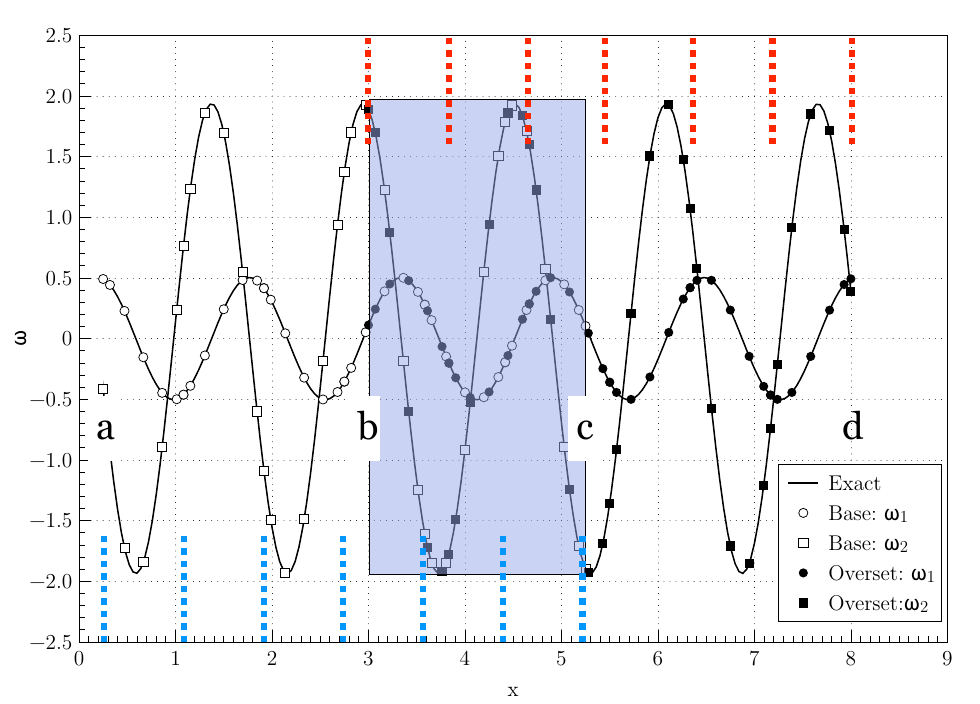}
   \caption{Exact and computed solutions for the sinusoidal problem, \eqref{eq:ExactPeriodicSolution}. Vertical dashed lines mark the element boundaries. The shaded area marks the overlap region.}
   \label{fig:SinusoidSoln}
\end{figure}

\subsection{Convergence with Equal Resolution}

The solutions with $\mmatrix{$\Sigma$}_v= \oneHalf \mmatrix A^+$ and $ \mmatrix{$\Sigma$}_u= |\mmatrix A^-|$ shown in Fig.~\ref{fig:SinusoidSoln} for six equally sized elements in each subdomain, are spectrally accurate. Fig.~\ref{fig:ErrorConvergence1,0.5} shows exponential convergence of the discrete $L_2$ errors $\inorm{\statevec U(T)-\statevecGreek \omega(T)}^2_{\Omega_u,N}$ and $\inorm{\statevec V(T) - \statevecGreek \omega(T)}^2_{\Omega_v,N}$ for $N \in[4,12]$. Along with the overset grid solutions, we have plotted the {\it optimal} error, which is what is computed when the overset grid coupling is removed and the exact solutions are used as the external states at $x=b$ and $x=c$. With this choice of penalty matrices, the overset domain is less accurate than the base. The base solution matches the optimal accuracy to graphical resolution.

\begin{figure}[htbp] 
   \centering
   \includegraphics[width=3.0in]{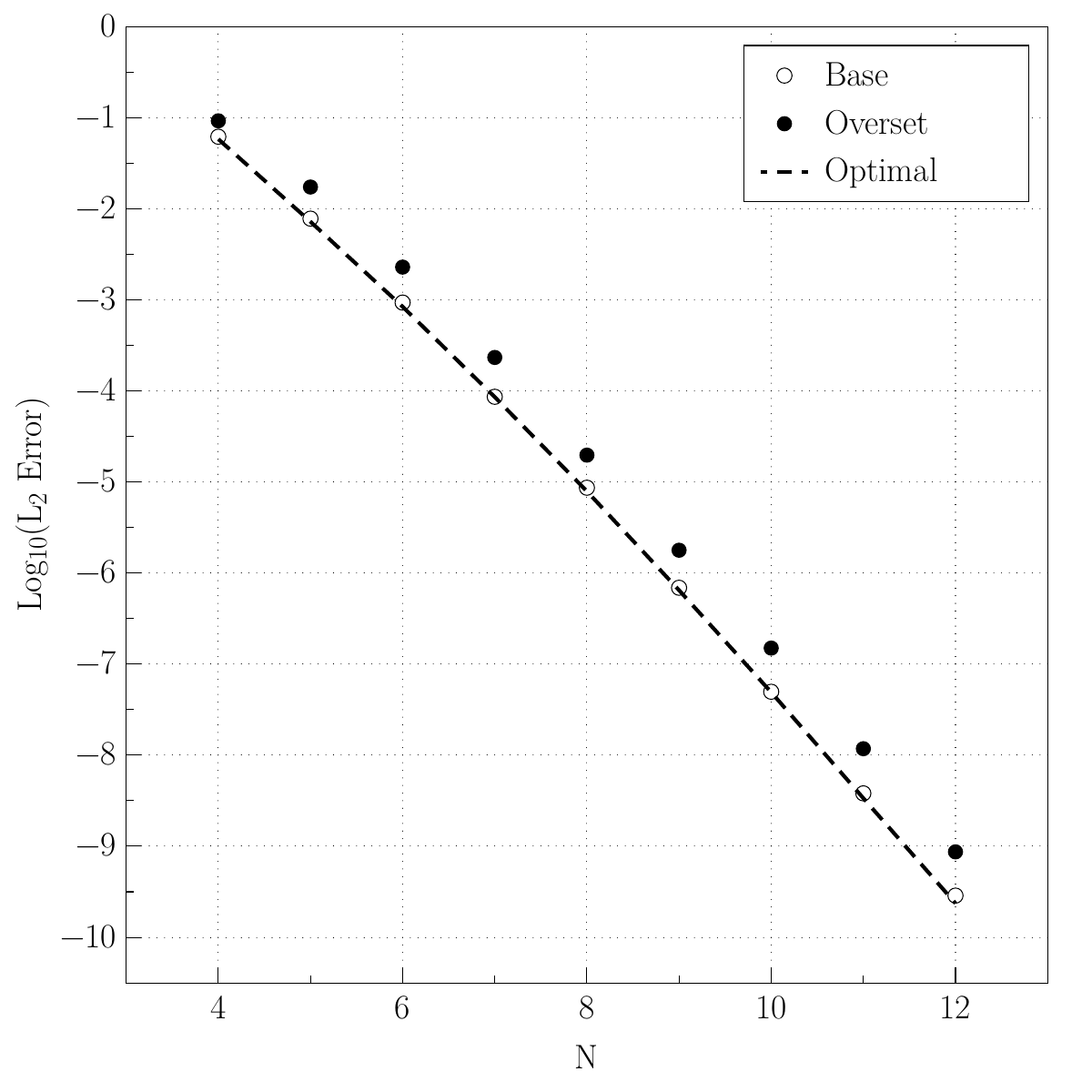}
   \caption{Error as a function of polynomial order.  $\mmatrix{$\Sigma$}_v= \oneHalf \mmatrix A^+$ and $ \mmatrix{$\Sigma$}_u= |\mmatrix A^-|$.}
   \label{fig:ErrorConvergence1,0.5}
\end{figure}


The parameters $\gamma_u$ and $\gamma_v$ can be varied to modify the error. For instance, Fig.~\ref{fig:ErrorConvergence1,1} shows that the errors match the optimal errors to graphical accuracy when the factors $\gamma_u = \gamma_v = 1$.
\begin{figure}[htbp] 
   \centering
   \includegraphics[width=3in]{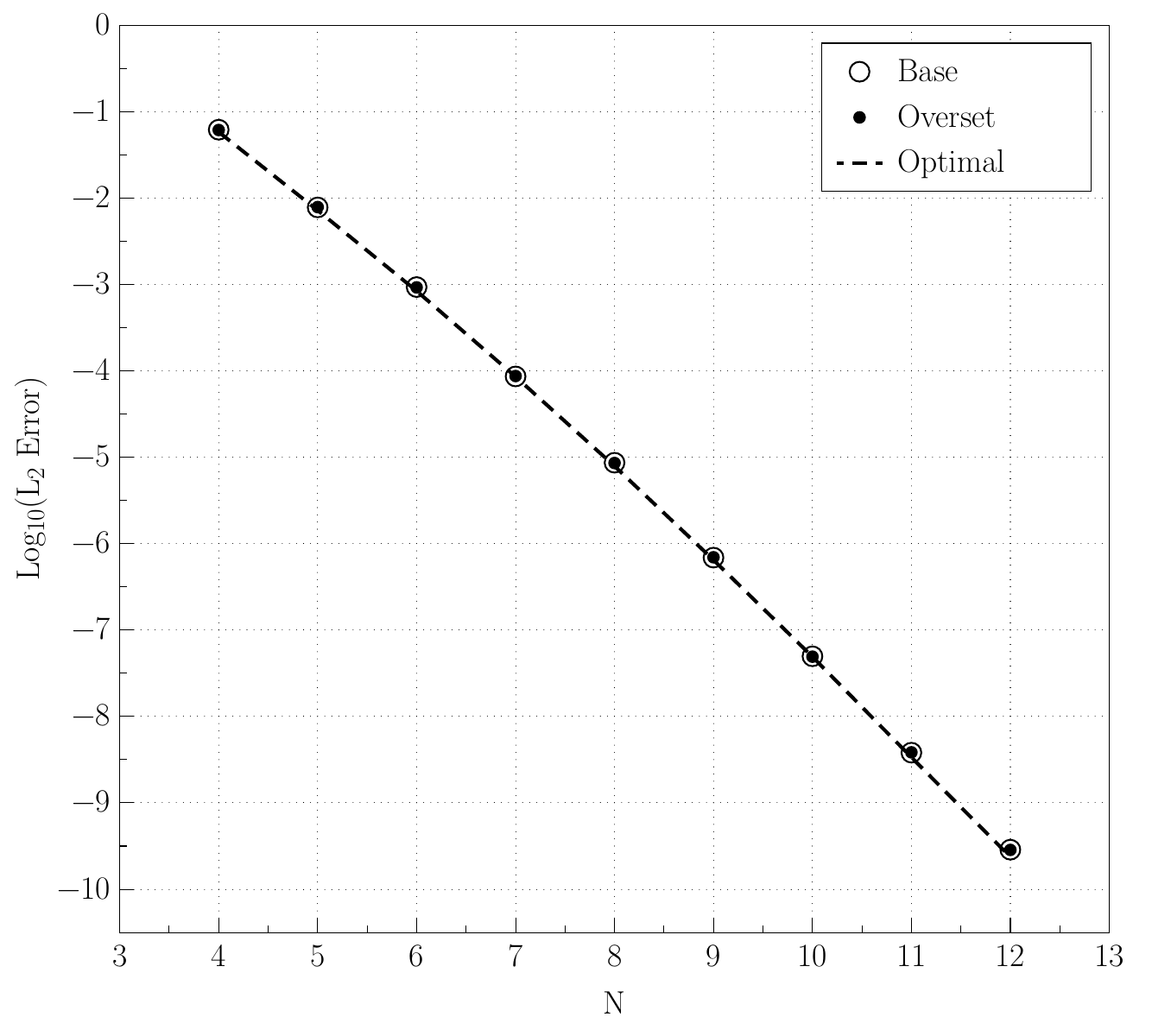}
   \caption{Same as Fig.~\ref{fig:ErrorConvergence1,0.5} but with $\mmatrix{$\Sigma$}_v=  \mmatrix A^+$ and $ \mmatrix{$\Sigma$}_u= |\mmatrix A^-|$.}
   \label{fig:ErrorConvergence1,1}
\end{figure}

In fact, for penalty matrices that satisfy \eqref{eq:SigmaDGSEMConditionsP}, the choice $\gamma_u = \gamma_v = 1$ is close to optimal. Fig.~\ref{fig:ErrorWithGammav} shows the total error $\left(\inorm{\statevec U(T)-\statevecGreek \omega(T)}^2_{\Omega_u,N}+\inorm{\statevec V(T) - \statevecGreek \omega(T)}^2\right)^\oneHalf$ for $\gamma_u=1$ as a function of $\gamma_v$, which has a minimum error that occurs around $\gamma_v = 0.9$. The global minimum occurs near (0.9,0.9), which is outside of the lower bounds \eqref{eq:SigmaDGSEMConditionsP} but not \eqref{eq:SigmaUBounds}. Near the optimal, the curve is fairly flat, and from $\gamma_v = 0.9$ to $\gamma_v = 1.0$ the total error varies by only $1.2\%$, making $\mmatrix{$\Sigma$}_v=  \mmatrix A^+$ and $ \mmatrix{$\Sigma$}_u= |\mmatrix A^-|$ a reasonable choice.
\begin{figure}[htbp] 
   \centering
   \includegraphics[width=3in]{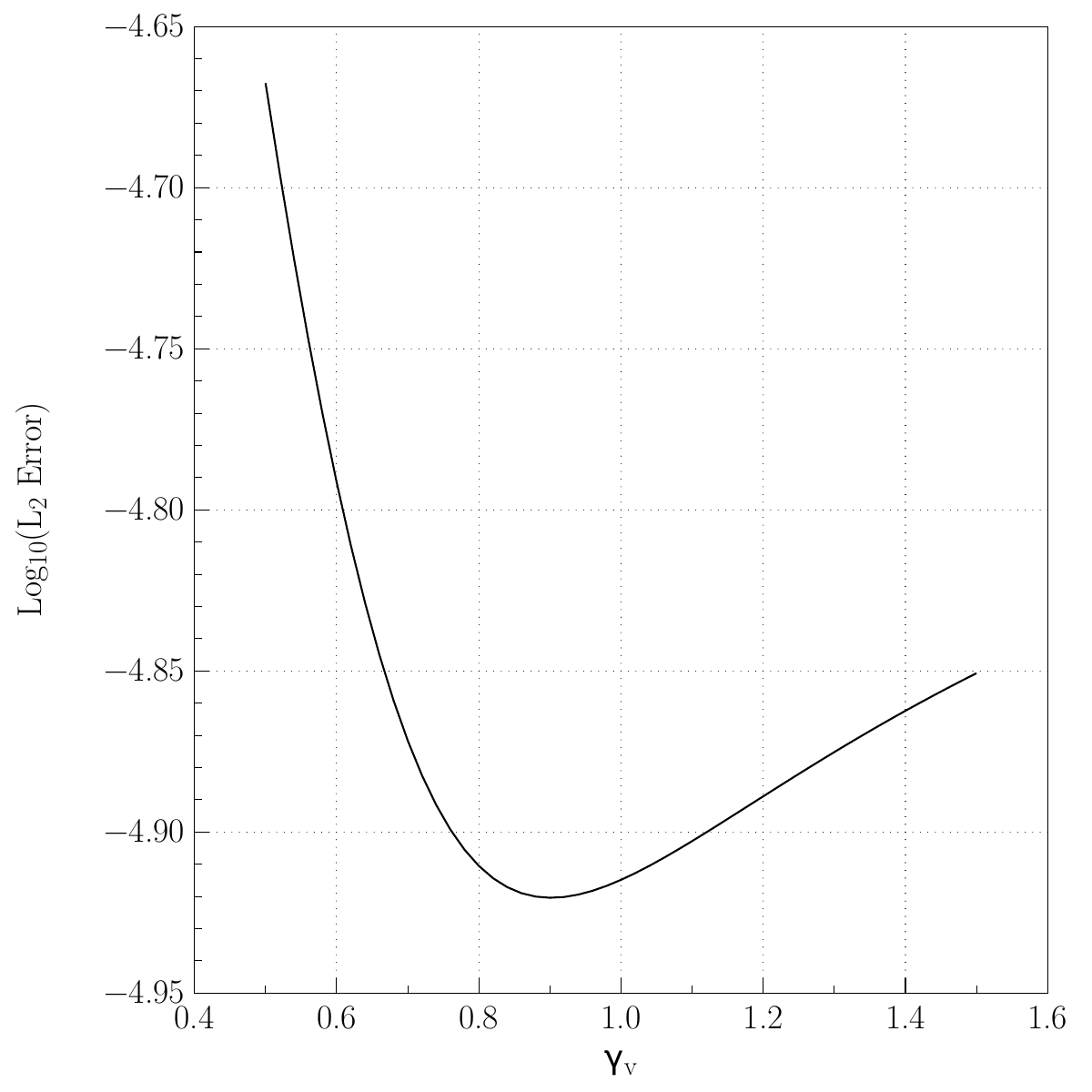}
   \caption{Total error as a function of $\gamma_v$.}
   \label{fig:ErrorWithGammav}
\end{figure}

\subsection{Convergence with Unequal Resolution}

One feature of overset grid methods is that the grids can have resolutions that differ within the overlap region. To assess the effects of different resolutions we use 12 elements in the base domain relative to the six in the overset domain. Fig.~\ref{fig:ErrorConvergenceNe=12,6} shows the error in each subdomain as a function of the polynomial order, plus the optimal error for each subdomain. The error in the less resolved subdomain remains near optimal, but it pollutes the better resolved one. Although there is a slight advantage observed in the base subdomain, the overset subdomain increases the error in the base subdomain by up to three orders of magnitude from the optimal.
The opposite behavior is observed if the number of elements is reversed.

\begin{figure}[htbp] 
   \centering
   \includegraphics[width=3in]{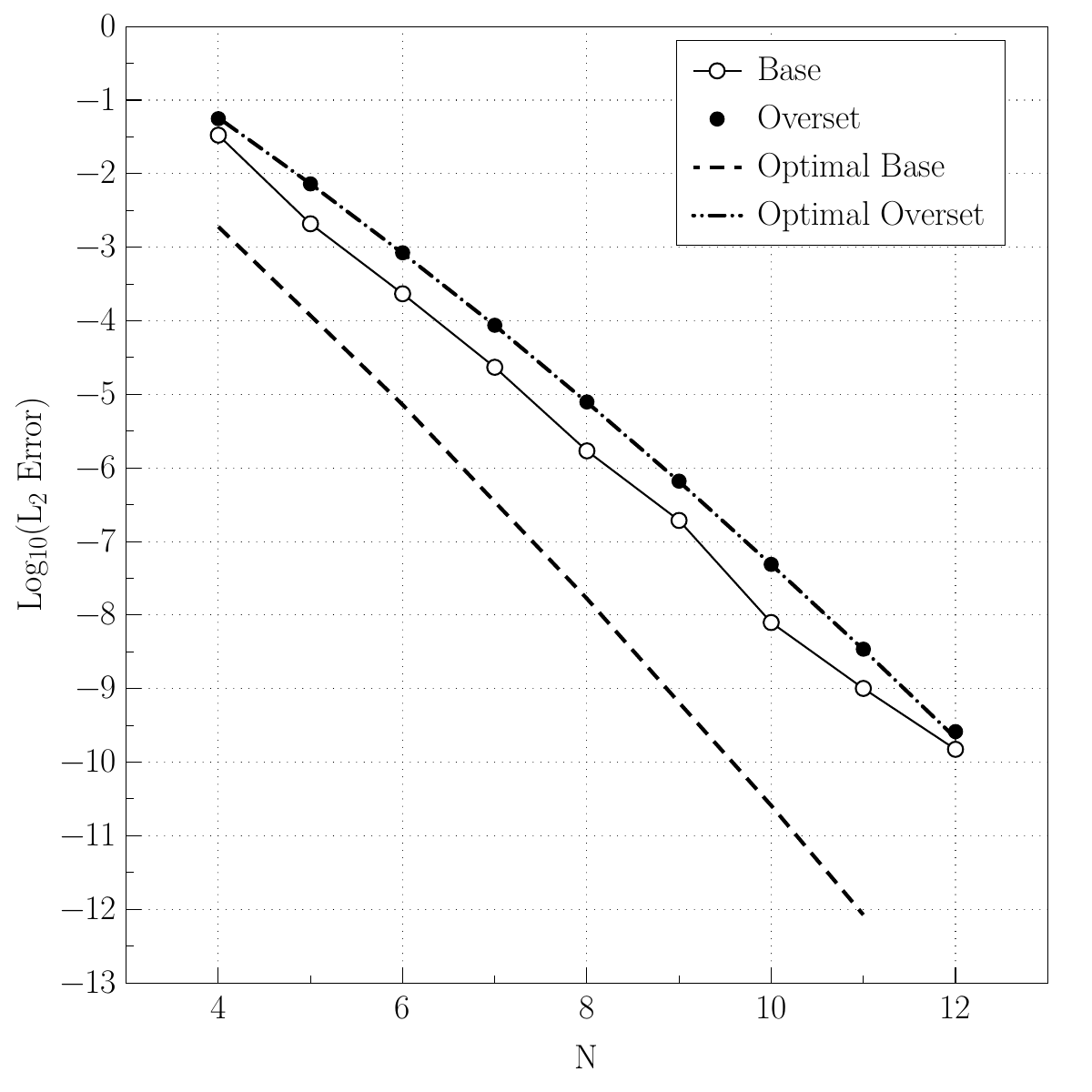}
   \caption{Error as a function of polynomial order where the base grid has twice as many elements as the overset.}
   \label{fig:ErrorConvergenceNe=12,6}
\end{figure}

\subsection{The Effect of Coupling within the Overlap Region}

The additional coupling \eqref{eq:WeakFormsFonFullDomainsWOvrlap} in the overlap region has the unique feature of adding dissipation proportional to the square of the difference of the solutions, \eqref{eq:OPDissipation} in that domain. The total amount of dissipation depends on the number of penalty points, $M$, and the amount of dissipation added at each point, $\varepsilon^m$. The parameters provide one the opportunity to tune the dissipation between the subdomains, though without much guidance.

As an example, we choose $\mmatrix{$\Sigma$}^m = \varepsilon \mmatrix I>0$ and $M=4$ for $N=8$ on a mesh where $N_u=12$ and $N_v=6$. Fig.~\ref{fig:EpsilonConvergence} shows the convergence for two values of $\varepsilon$ compared to no overlap penalty applied ($\varepsilon=0$). First, note that the convergence rate is not affected, as expected. Second, we see that the penalty parameter has little effect on the magnitude of the error at these high orders. One can optimize the error as a function of $\varepsilon$, as seen in Fig.~\ref{fig:OverlapPenaltyError.pdf}, where we vary $\varepsilon$ between zero and two. However, the exercise shows that the error is insensitive (with a variation of less than 1.5\%) to the overlap penalty terms. This conclusion is borne out over various polynomial orders. The results suggest that adding the overlap penalties serve to enhance dissipation and hence stability, but not accuracy, as was postulated in \cite{KOPRIVA2022110732}, the likely reason being that the coupling at the overlap boundaries is already affecting the accuracy in the same way, c.f. \eqref{eq:WeakFormsFonFullDomainsWOvrlap}.

\begin{figure}[htbp] 
   \centering
   \includegraphics[width=3in]{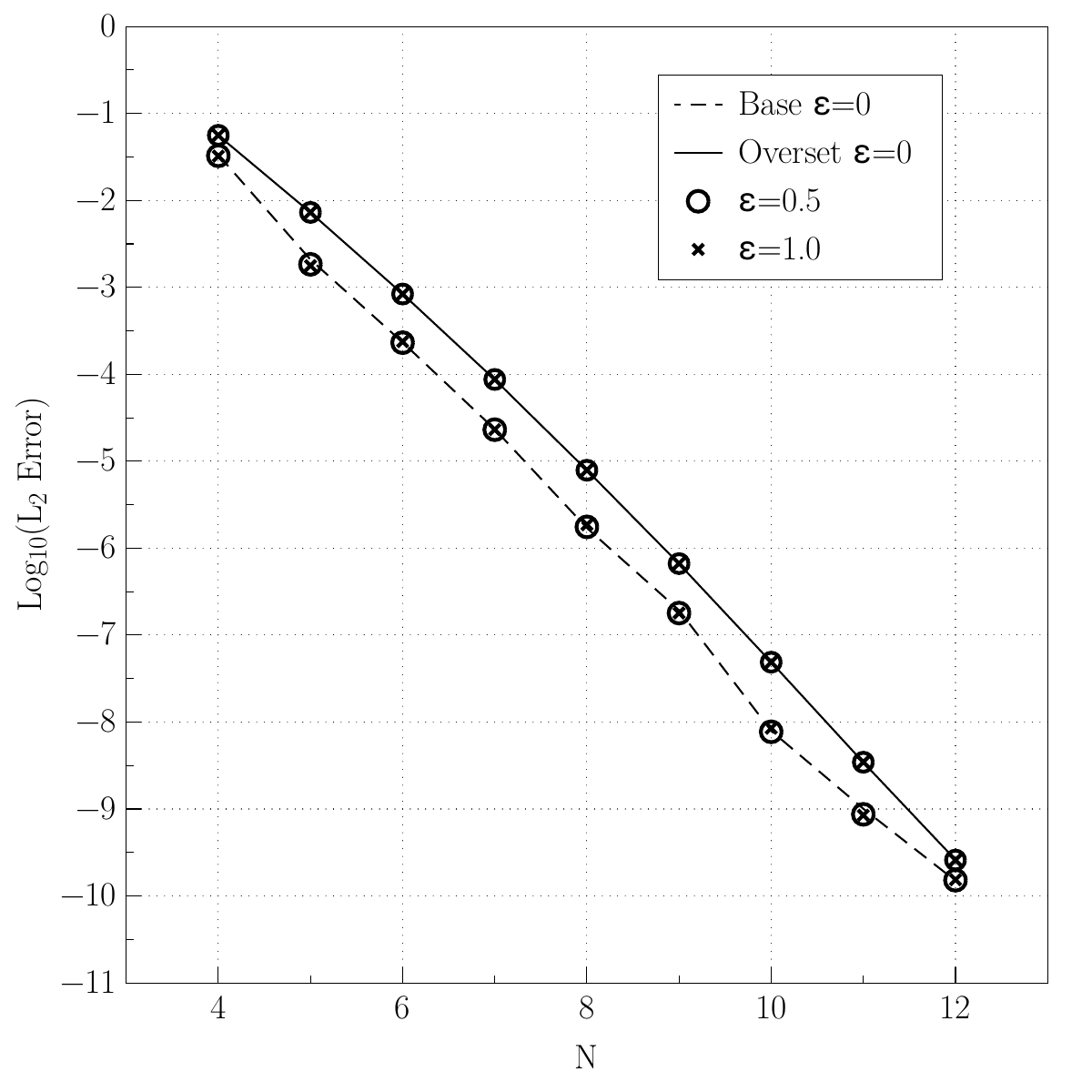}
   \caption{Convergence as a function of overlap penalty weight where the base grid has twice as many elements as the overset.}
   \label{fig:EpsilonConvergence}
\end{figure}
\begin{figure}[htbp] 
   \centering
   \includegraphics[width=3in]{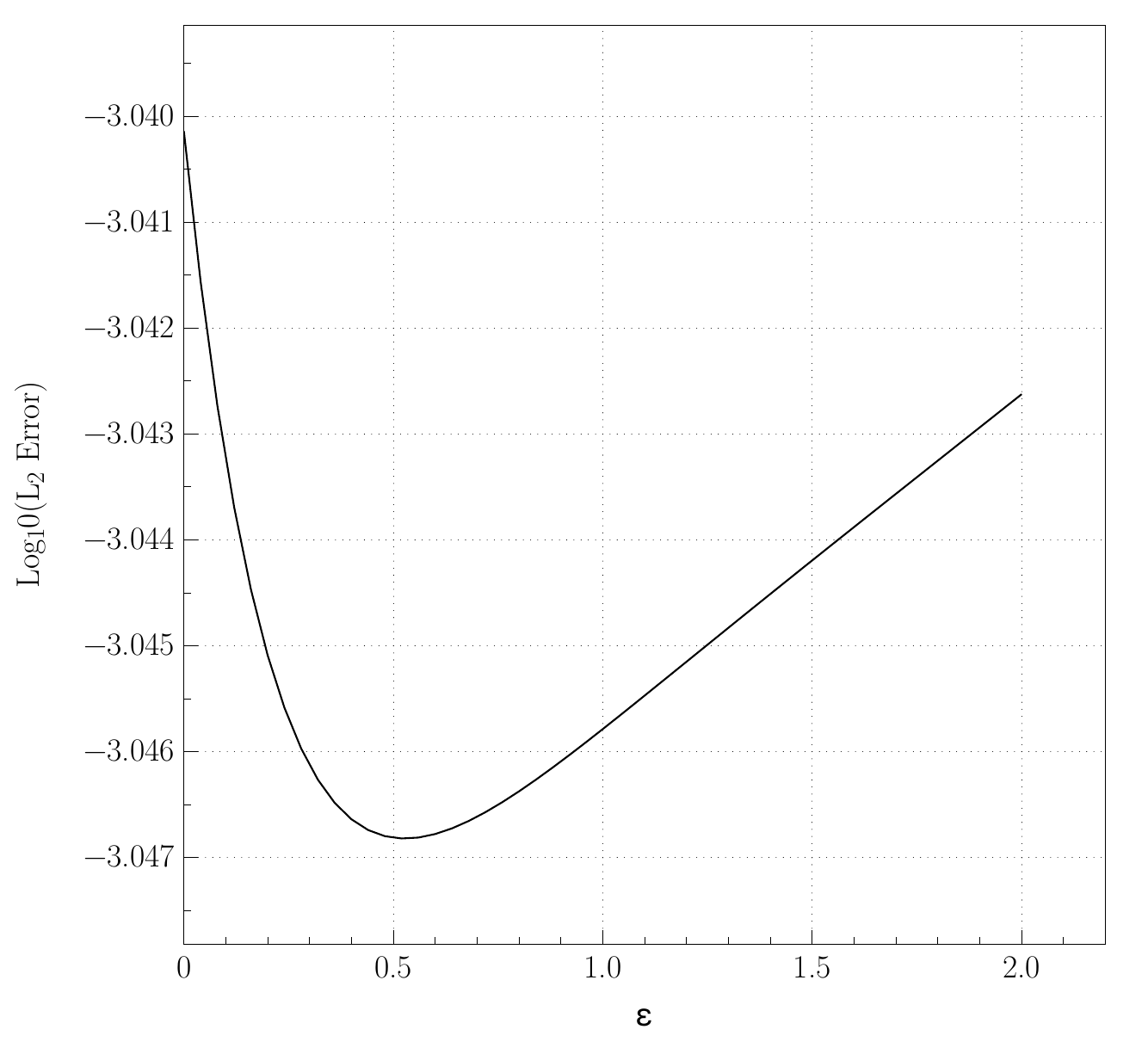}
   \caption{Error as a function of overlap penalty weight.}
   \label{fig:OverlapPenaltyError.pdf}
\end{figure}

\section{Summary and Discussion}

The analysis of stability of overset grid methods has been a challenge ever since the methods were first introduced. Here, we used the energy method and eigenvalue studies to examine the stability of characteristic approximations of the DGSEM for hyperbolic scalar equations and systems. The energy bounds show growth with time, polynomial order, and the inverse of the element size, and so they cannot be said to be energy stable, even though they are for the original problem. The same behavior was observed for the approximation of the well-posed penalty formulation developed in \cite{KOPRIVA2022110732}. On the other hand, for fixed polynomial order and mesh, the growth in time can be at most exponentially fast. \textcolor{black}{It was not possible to show stability of the approximation of the penalty formulation in the same way as showing energy boundedness of the PDE because the latter required the use of integration-by-parts over subintervals of the domain, whereas the DGSEM (as well as other SBP approximations) has the summation-by-parts property only over an entire element. This suggests that a global or elementwise SBP property is not enough to prove stability of the overset problem, but that if approximations that have a local SBP property can be found, those might be provably stable.}

The eigenvalue studies in \ref{AppendixB} show that for systems the characteristic overset grid coupling itself is destabilizing, allowing for positive growth eigenvalues in the discrete system. This differs from the scalar case, and shows the pitfalls of extrapolating scalar results. The presence of unstable eigenvalues is consistent with the growth terms found in the energy analysis. It follows that dissipation is needed to keep solutions bounded over long times, which has been known from practical experience. The well-posed penalty formulation has four sources of dissipation when upwinding is used in the numerical flux: Dissipation due to physical boundary conditions, which allows energy to leave the physical domain, inter-element dissipation that is a characteristic of discontinuous Galerkin methods, a dissipative penalty at the ends of the overlap domain, and, finally, dissipative penalties applied at arbitrary points within the overlap. All these dissipative terms depend on the jumps in the solutions between elements, the boundary conditions, and between the solutions of the base and overset domains.

The overlap penalty dissipation introduced in \cite{KOPRIVA2022110732} is novel, and could be applied to any overset grid method. It produces dissipation that does not degrade the rate of convergence, but it can be used as a stabilizing factor.

\section*{Acknowledgments}

 This work was supported by grants from the Simons Foundation (\#426393, \#961988, David Kopriva).  Andrew Winters were supported by Vetenskapsr{\aa}det, Sweden (award no.~2020-03642 VR). Jan Nordstr\"{o}m was supported by Vetenskapsr{\aa}det, Sweden (award no.~2021-05484 VR) and the University of Johannesburg.

\section*{CRediT authorship contribution statement}

\textbf{David A. Kopriva}: 
Conceptualization, Data curation, Investigation, Methodology,
Software, Validation, Visualization, Writing --- original draft.
\textbf{Andrew R. Winters}:
Conceptualization, Methodology,
Writing --- original draft.
\textbf{Jan Nordstr\"{o}m}:
Conceptualization, Methodology,
Writing --- original draft.

\section*{Declaration of competing interest}

The authors declare that they have no known competing financial
interests or personal relationships that could have appeared to
influence the work reported in this paper.

\section*{Data availability}


Data will be made available upon request.

\appendix
\section{Central Numerical Flux Bounds}\label{AppendixA}

The approximation with the central numerical flux,
\begin{equation}
\statevec F^* = \avg{\statevec F},
\label{eq:CentralNumFlux2}
\end{equation}
differs from the upwind flux, $\avg{\statevec F} - \oneHalf |\mmatrix A|\jump{\statevec U}$, by missing the jump term. The central flux produces the neutral state,  from which we can isolate the behavior of the overset grid coupling from the stabilizing effect of the dissipation produced by the upwind flux.

The energy of a DGSEM approximation in a domain is governed by
\begin{equation}
\oneHalf\frac{d}{dt} \sum_{k=1}^{K_u} \inorm{\statevec U}^2_{\underline{e}^k} + \statevec U^T \left.\left(\statevec F^* -\oneHalf \statevec F\right)\right|_{left}^{right} = \sum_{k=1}^{{K_u}-1}\left\{ \jump{\statevec U}^T\statevec F^* - \oneHalf \jump{\statevec U^T\statevec F}\right\}.
\label{eq:UDomainEnergyApp}
\end{equation}

Using the central flux at the element interfaces, then, the interface contributions on the right of \eqref{eq:UDomainEnergyApp} vanish and are non-dissipative. At the left boundary,
\begin{equation}
BTL = \statevec U^T \left(\statevec F^*(\statevec g_L,\statevec U) -\oneHalf \statevec F(\statevec U)\right) =  \oneHalf\statevec U^T \mmatrix A \statevec g_L .
\end{equation}
Similarly,
\begin{equation}
BTR = \statevec U^T \left(\statevec F^*(\statevec U, \statevec g_R) -\oneHalf \statevec F(\statevec U)\right) = \oneHalf\statevec U^T \mmatrix A \statevec g_R
\end{equation}


Then we start with the following result about the DGSEM on a single domain,
\begin{thm}
For the system of PDEs \eqref{eq:PDESystem} for the normal, single domain problem, the DGSEM with the central numerical flux \eqref{eq:CentralNumFlux2} is dissipation free and neutrally stable.
\end{thm}
\begin{proof}
Both $BTR=0$ and $BTL=0$ when $\statevec g_L = \statevec g_R = \statevec 0$. Then from \eqref{eq:UDomainEnergyApp},
\begin{equation}
\oneHalf\frac{d}{dt} \inorm{\statevec U}^2_{\Omega,N} = \oneHalf\frac{d}{dt} \sum_{k=1}^{K_u} \inorm{\statevec U}^2_{N,\underline{e}^k} = 0.
\label{eq:EnergyEstimateCentralMD}
\end{equation}
\end{proof}

On the other hand, when the external states are not zero (recalling that $\mmatrix A = \mmatrix A^T$),
\begin{equation}
\left(\statevec U - \mmatrix A\statevec g \right)^2 = |\statevec U|^2 + |\mmatrix A \statevec g|^2 - 2\statevec U^T \mmatrix A \statevec g\ge 0.
\end{equation}
Therefore,
\begin{equation}
BTL = \oneHalf \statevec U^T \mmatrix A \statevec g_L \le \frac{1}{4}\left\{ |\statevec U|^2 + |\mmatrix A \statevec g_L|^2\right\}
\end{equation}
and
\begin{equation}
BTR = \oneHalf \statevec U^T \mmatrix A \statevec g_R \le \frac{1}{4}\left\{ |\statevec U|^2 + |\mmatrix A \statevec g_R|^2\right\}
\end{equation}
Then we find that
\begin{thm}
The DGSEM is not stable with the central numerical flux when the boundary data are not zero.
\end{thm}
\begin{proof}
\begin{equation}
\begin{split}
\frac{d}{dt}\inorm{\statevec U}^2_{\Omega_u,N} &= 2BTL- 2BTR \le \oneHalf \left\{ |\statevec U(x_L)|^2 + |\mmatrix A \statevec g_L|^2\right\} + \oneHalf\left\{ |\statevec U(x_R)|^2 + |\mmatrix A \statevec g_R|^2\right\}
\\&\le \inorm{\statevec U}_\infty^2 + \oneHalf\left\{  |\mmatrix A \statevec g_L|^2 +  |\mmatrix A \statevec g_R|^2\right\}.
\end{split}
\end{equation}
The time derivative of the energy is seen to depend on the solution when the central numerical flux is used, and not just on the data. That dependency is made dissipative and therefore removed from the bound when the dissipative terms from the upwind numerical flux are included, see \eqref{eq:BTLSystem}.

As before, we can bound the maximum norm of the solution, so that
\begin{equation}
\frac{d}{dt}\inorm{\statevec U}^2_{\Omega_u,N} \le C\Delta x^{-1}N^2\inorm{\statevec U}^2_{\Omega_u,N} + \oneHalf\left\{  |\mmatrix A \statevec g_L|^2 +  |\mmatrix A \statevec g_R|^2\right\}.
\end{equation}
Therefore,
\begin{equation}
\inorm{\statevec U}^2_{\Omega_u,N} \le e^{C\Delta x^{-1}N^2T}\inorm{\statevec \omega_0}^2_{\Omega_u,N} + \oneHalf\int_0^T e^{C\Delta x^{-1}N^2(t-s)} \left\{  |\mmatrix A \statevec g_L|^2 +  |\mmatrix A \statevec g_R|^2\right\}\,\mathrm{d}s.
\end{equation}
\end{proof}

The overset grid problem sets $\statevec g_L = 0$ and $\statevec g_R = \statevec V(c)$ for the base domain and  $\statevec g_R = 0$ and $\statevec g_L = \statevec U(b)$
for the overset domain. Therefore,
\begin{equation}
\begin{split}
\frac{d}{dt}\inorm{\statevec U}^2_{\Omega_u,N} &\le \oneHalf\statevec V(c)^T\mmatrix A^2\statevec V(c) + C\Delta x^{-1}N^2\inorm{\statevec U}^2_{\Omega_u,N}\le \oneHalf\rho^2(\mmatrix A)|{\statevec V(c)}|^2 +  CN^2\inorm{\statevec U}^2_{\Omega_u,N}
\\&\le C\Delta x^{-1}N^2\left\{\rho^2(\mmatrix A)\inorm{\statevec V}^2_{\Omega_v,N} + \inorm{\statevec U}^2_{\Omega_u,N}\right\}
\end{split}
\end{equation}
and
\begin{equation}
\frac{d}{dt} \inorm{\statevec V}^2_{\Omega_v,N}\le C\Delta x^{-1}N^2\left\{\rho^2(\mmatrix A)\inorm{\statevec U}^2_{\Omega_u,N} + \inorm{\statevec V}^2_{\Omega_v,N}\right\}.
\end{equation}
Adding the two together,
\begin{equation}
\frac{d}{dt} \left\{ \inorm{\statevec U}^2_{\Omega_u,N} + \inorm{\statevec V}^2_{\Omega_v,N} \right\}
\le C\Delta x^{-1}N^2(\rho^2(\mmatrix A)+1)\left\{\inorm{\statevec V}^2_{\Omega_v,N} + \inorm{\statevec U}^2_{\Omega_u,N} \right\},
\end{equation}
so
\begin{equation}
 \inorm{\statevec U}^2_{\Omega_u,N} + \inorm{\statevec V}^2_{\Omega_v,N} \le e^{C\Delta x^{-1}N^2\left(\rho^2\left(\mmatrix A\right)+1\right)T}\left\{ \inorm{\statevec \omega_0}^2_{\Omega_u,N} + \inorm{\statevec \omega_0}^2_{\Omega_v,N} \right\}
\end{equation}
Therefore we have proved
\begin{thm}
The overset grid problem is not stable in the combined energy norm, but is bounded by the initial data for fixed $N$, $\Delta x$, and $T$.
\end{thm}


\section{Matrix Structure of the Overset Grid System}\label{AppendixB}

Here, we examine the structure of the matrices one gets for the spatial terms when approximating the overset grid problems with the DGSEM, again using the central numerical flux to isolate the effects of the coupling from the dissipation introduced by the upwind flux. The point is to show in terms of the eigenvalues that the overset grid coupling itself is destabilizing, making dissipation necessary for a dynamically stable (fixed grid and $N$, as $t$ increases) approximation. A study of the eigenvalues for general SBP approximations can be found in \cite{Sharan:2016rz}. The derivations here are specific to the DGSEM, and provide alternate proofs to reach the same conclusions.

\begin{rem}
We refer to eigenvalues that have positive real parts as ``unstable". They are ``stable" otherwise.
\end{rem}
We start with the approximation of the scalar problem
\begin{equation}
u_t + u_x = 0, \quad x\in [-1,1].
\end{equation}
The DGSEM approximation is
\begin{equation}
\iprodN{U_t,\phi} + \phi(F^* - F)|_{-1}^{1} + \iprodN{\phi,U_x} = 0,
\end{equation}
where $F=U$ and $F^*$ is the numerical flux. Now, the LGL quadrature satisfies the summation-by-parts property,
\begin{equation}
\iprodN{\phi,U'} = \phi U|_{-1}^1 - \iprodN{\phi',U}.
\end{equation}
Therefore,
\begin{equation}
\iprodN{\phi,U'} = \oneHalf\left\{ \phi U|_{-1}^1 + \iprodN{\phi,U'}- \iprodN{\phi',U} \right\},
\end{equation}
which allows us to write the equivalent split form approximation
\begin{equation}
\iprodN{U_t,\phi} +\left. \phi\left( F^* - \oneHalf F\right)\right|_{-1}^{1}+\oneHalf\left\{\iprodN{\phi,U'}- \iprodN{\phi',U} \right\}=0.
\end{equation}
To select the nodal values we take $\phi = \ell_j$. Then (see \cite{Kopriva:2009nx})
\begin{equation}
w_j\dot U_j +\delta_{jN}\left(U^*(U_N,g_R) - \oneHalf U_N   \right) - \delta_{j0}\left(U^*(g_L,U_0) - \oneHalf U_0   \right) + \oneHalf\left\{ \sum_{n=0}^N U_n\left(w_jD_{jn} - w_nD_{nj} \right)\right\}=0
\label{eq:PtwiseScalarU}
\end{equation}
where $D_{jn} = \ell'_n(x_j)$ is the derivative matrix.

We discuss three specific cases for the scalar equation. The first is the single element problem. The second shows that one gets the same results for multiple elements. The final case is the overset grid problem for the scalar equation.
\\

{\bf Case 1. DGSEM on a Single Element.} The first situation is the use of the central numerical flux, $U^*(r,s) = \oneHalf ( r+s)$ on a single element. Then
\begin{equation}
w_j\dot U_j +\oneHalf \delta_{jN}g_R - \oneHalf\delta_{j0}g_L + \oneHalf\left\{ \sum_{n=0}^N U_n\left(w_jD_{jn} - w_nD_{nj} \right)\right\}=0.
\end{equation}
We can write the problem in matrix-vector form as
\begin{equation}
\mmatrix W_e\frac{d}{dt} \vec U =  \oneHalf\mmatrix S\vec U + \oneHalf \vec g,
\label{eq:SDMatrixEquation}
\end{equation}
where $S_{jn} = -\left(w_jD_{jn} - w_nD_{nj} \right) = -S_{nj}$ is a skew-symmetric matrix and $\mmatrix W_e = \mathrm{diag}(w_j) >0$.

Then we have,
\begin{thm}
The eigenvalues of $\mmatrix W_e^{-1}\mmatrix S$ lie on the imaginary axis.
\label{thm:EVsOfSingleDomain}
\end{thm}
\begin{proof}
The matrix $\mmatrix S$ has purely imaginary eigenvalues by virtue of being skew-symmetric. Since $\mmatrix W_e > 0$, $\mmatrix W_e^{-1} = \sqrt{\mmatrix W_e^{-1}}\sqrt{\mmatrix W_e^{-1}}$. The matrix $\sqrt{\mmatrix W_e^{-1}}\sqrt{\mmatrix W_e^{-1}}\mmatrix S$ has the same eigenvalues as $\sqrt{\mmatrix W_e^{-1}}\mmatrix S\sqrt{\mmatrix W_e^{-1}}$, which is skew-symmetric because $\mmatrix S$ is skew symmetric. Therefore, the result holds.
\end{proof}

Fig.~\ref{fig:DGSEMCentrEvs.pdf} shows the eigenvalues for the DGSEM with a single element on the interval $[-1,1]$ of polynomial order $N=16$. The eigenvalues, as proved, are all along the imaginary line, with maximum positive real parts of $1.65\times 10^{-14}$, i.e., on the order of rounding error.\\
\begin{figure}[htbp] 
   \centering
   \includegraphics[width=3in]{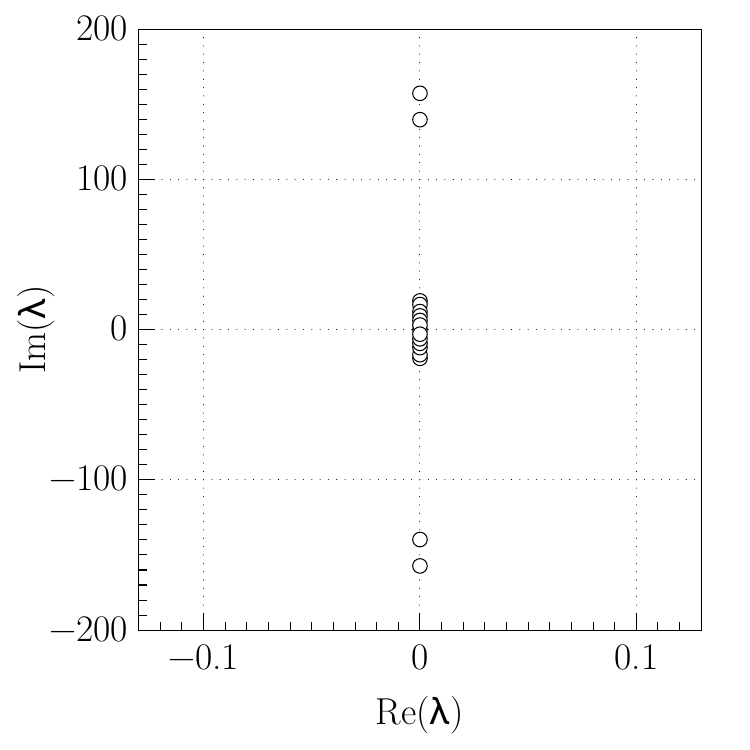}
   \caption{Eigenvalues of the DGSEM, $\mmatrix W_e^{-1}\mmatrix S$, with one element and the central numerical flux for $N=16$.}
   \label{fig:DGSEMCentrEvs.pdf}
\end{figure}


{\bf Case 2. Single Domain with Multiple Elements.}
The eigenvalues remain on the imaginary axis when the central flux is used at the boundaries and element interfaces, consistent with the energy estimate \eqref{eq:EnergyEstimateCentralMD}. With multiple elements, the matrix $\mmatrix S$ in the matrix system \eqref{eq:SDMatrixEquation} is replaced by the block tri-diagonal matrix
\begin{equation}
\mmatrix Q = \left[\begin{array}{ccccc}
\mmatrix S & \mmatrix C &  &  &  \\
-\mmatrix C^T & \mmatrix S & \mmatrix C &  &  \\ 
& \ddots & \ddots & \ddots &  \\ 
&  & -\mmatrix C^T & \mmatrix S & \mmatrix C\\
& & & -\mmatrix C^T & \mmatrix S \\
\end{array}\right],
\end{equation}
where
\begin{equation}
\mmatrix C = \left[\begin{array}{cccc}
0 & 0 & \dots & 0 \\\vdots &  &  & \vdots \\0 & 0 & \dots & 0
\\1 & 0 & \dots & 0\end{array}\right]
\end{equation}
is the coupling matrix.

For multiple elements, the matrix $\mmatrix Q$ is also skew-symmetric and therefore has purely imaginary eigenvalues.
It is therefore sufficient to consider only a single element per domain in the following example.\\

{\bf Case 3. Overset Grid with Two Elements.} The final scalar case is the overset grid problem. We consider here two elements only, with domains $\Omega_u = [-1,1]$ and $\Omega_v = [o,o+2]$, where $o\in(-1,1)$, so that no metrics are involved. In this case,
\begin{equation}
V(-1) = U(b) = \sum_{n=0}^N U_n\ell_n(b),
\end{equation}
where, as before, we do not explicitly include the transformation to reference space when writing the Lagrange interpolating polynomials.
Then
\begin{equation}
w_j\dot V_j = \oneHalf \delta_{j0}\sum_{n=0}^N U_n\ell_n(b) - \oneHalf \delta_{jN} g_R+\oneHalf \left(\mmatrix S\vec V\right)_j.
\end{equation}
In matrix-vector form, the system is
\begin{equation}
\mmatrix W \frac{d}{dt}\left[\begin{array}{c}\vec U \\\vec V\end{array}\right] = \oneHalf \left[\begin{array}{cc}\mmatrix S & \mmatrix 0 \\\mmatrix B & \mmatrix S\end{array}\right]\left[\begin{array}{c}\vec U \\\vec V\end{array}\right] + \vec g = \oneHalf\mmatrix Q\left[\begin{array}{c}\vec U \\\vec V\end{array}\right] + \vec g
\end{equation}
where
\begin{equation}
\mmatrix B = \left[\begin{array}{ccc}\ell_0(b) & \dots & \ell_N(b) \\0 & \dots & 0 \\\vdots & \ddots & 0 \\0 & \dots & 0\end{array}\right]
\end{equation}
is the matrix that couples the base grid to the overset grid.

For the scalar problem, the eigenvalues do not change when coupling two domains (Cf. \cite{Sharan:2016rz},\cite{SHARAN2018199}) as shown in
\begin{thm}
The DGSEM applied with the central numerical flux to the scalar overset grid problem has purely imaginary eigenvalues.
\end{thm}
\begin{proof}
Since $\mmatrix W$ is diagonal, the block matrix
\begin{equation}
\hat{\mmatrix Q} = \mmatrix W^{-1}\mmatrix Q =  \left[\begin{array}{cc}W_e^{-1}\mmatrix S & \mmatrix 0 \\\frac{1}{w_0}\mmatrix B & W_e^{-1}\mmatrix S\end{array}\right]\left[\begin{array}{c}\vec U \\\vec V\end{array}\right]
\end{equation}
is block triangular so its determinant is $\det(\hat Q) = \det(W_e^{-1}\mmatrix S)\det(W_e^{-1}\mmatrix S)$. Therefore the eigenvalues of $\hat{\mmatrix Q}$ are the eigenvalues of $W_e^{-1}\mmatrix S$, which by Thm. \ref{thm:EVsOfSingleDomain} are purely imaginary.
\end{proof}

%

Unfortunately, the scalar problem is a special case. For systems, the eigenvalues are not so well-behaved if one does not diagonalize the system first and add dissipation through upwinding \cite{SHARAN2018199}. They will generally have both positive and negative real parts.\\

{\bf Case 4. Overset Grid for the System.}  The final example is the case of the system, \eqref{eq:1DWaveEquationSystem}, for one element each per subdomain.

 For the system, and with the central numerical flux, the DGSEM approximations for the base and overset grids are
\begin{equation}
\begin{gathered}
w_j\Delta x_u\dot{\statevec U}_j + \oneHalf\delta_{jN}\mmatrix A\statevec V(c) + \oneHalf \mmatrix A\left(\mmatrix S\vec{\statevec U} \right)_j = 0\\
w_j\Delta x_v\dot{\statevec V}_j - \oneHalf\delta_{j0}\mmatrix A\statevec U(b) + \oneHalf \mmatrix A\left(\mmatrix S\vec{\statevec V} \right)_j = 0,
\end{gathered}
\end{equation}
where $\mmatrix S$ is now a block matrix with blocks whose size is the rank of $\mmatrix A$. As before, the donor values of the solutions are interpolated from the donor subdomains, so
\begin{equation}
\begin{gathered}
\statevec V(c) = \sum_{n=0}^N\ell_n(c)\statevec V_n \\
\statevec U(b) = \sum_{n=0}^N\ell_n(b)\statevec U_n.
\end{gathered}
\label{eq:OverlapPointsSystem}
\end{equation}
Then in matrix-vector form,
\begin{equation}
\hat{\mmatrix W} \frac{d}{dt}\left[\begin{array}{c}\vec {\statevec U} \\\vec {\statevec V}\end{array}\right] = \oneHalf \left[\begin{array}{cc}\mmatrix S & \mmatrix O \\\mmatrix B & \mmatrix S\end{array}\right]\left[\begin{array}{c}\vec {\statevec U} \\\vec {\statevec V}\end{array}\right] + \vec {\statevec g} = \oneHalf\mmatrix Q\left[\begin{array}{c}\vec {\statevec U} \\\vec {\statevec V}\end{array}\right] + \vec {\statevec g}
\end{equation}
where
\begin{equation}
\mmatrix B = \left[\begin{array}{ccc}\ell_0(b)\mmatrix A & \dots & \ell_N(b)\mmatrix A\\\mmatrix 0 & \dots & \mmatrix 0 \\\vdots & \ddots & \mmatrix 0 \\\mmatrix 0 & \dots & \mmatrix 0\end{array}\right],\quad  \mmatrix O= -\left[\begin{array}{ccc}\mmatrix 0 & \dots & \mmatrix 0 \\\vdots & \ddots & \mmatrix 0 \\\mmatrix 0 & \dots & \mmatrix 0 \\\ell_0(c)\mmatrix A & \dots & \ell_N(c)\mmatrix A\\\end{array}\right].
\end{equation}
The matrix $\hat{\mmatrix W}$ includes the element sizes, which remains diagonal and positive definite, so that it does not change the eigenvalues. The matrices $\mmatrix S$, $\mmatrix B$ and $\mmatrix O$ are each block matrices whose blocks have the rank of the system coefficient matrix, $\mmatrix A$.

For the system of equations, where there is now two-way  coupling between the domains, the matrix $\mmatrix Q$ is no longer block triangular, nor is it skew-symmetric. It is now indefinite so eigenvalues are no longer guaranteed to be on the imaginary axis, and can have positive real parts. As examples, Fig.~\ref{fig:CentChimeraSystemEvs} shows the $N=5$ eigenvalues for a base domain of $\Omega_u = [0,2]$ with three overset domains, showing that the eigenvalues have no set pattern, and in two of the cases have unstable ones.
\begin{figure}[htbp] 
   \centering
   \includegraphics[width=1.85in]{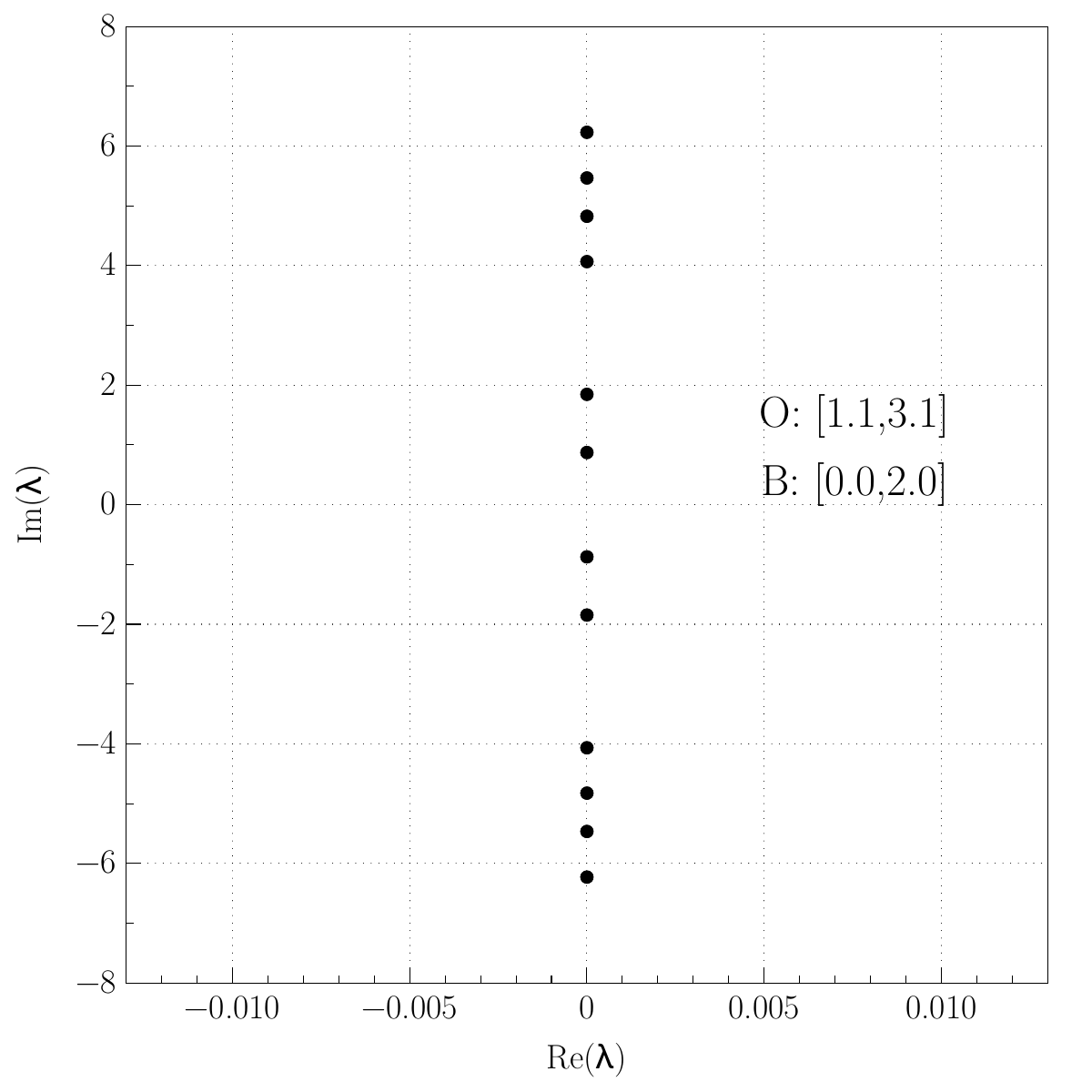}
   \includegraphics[width=1.85in]{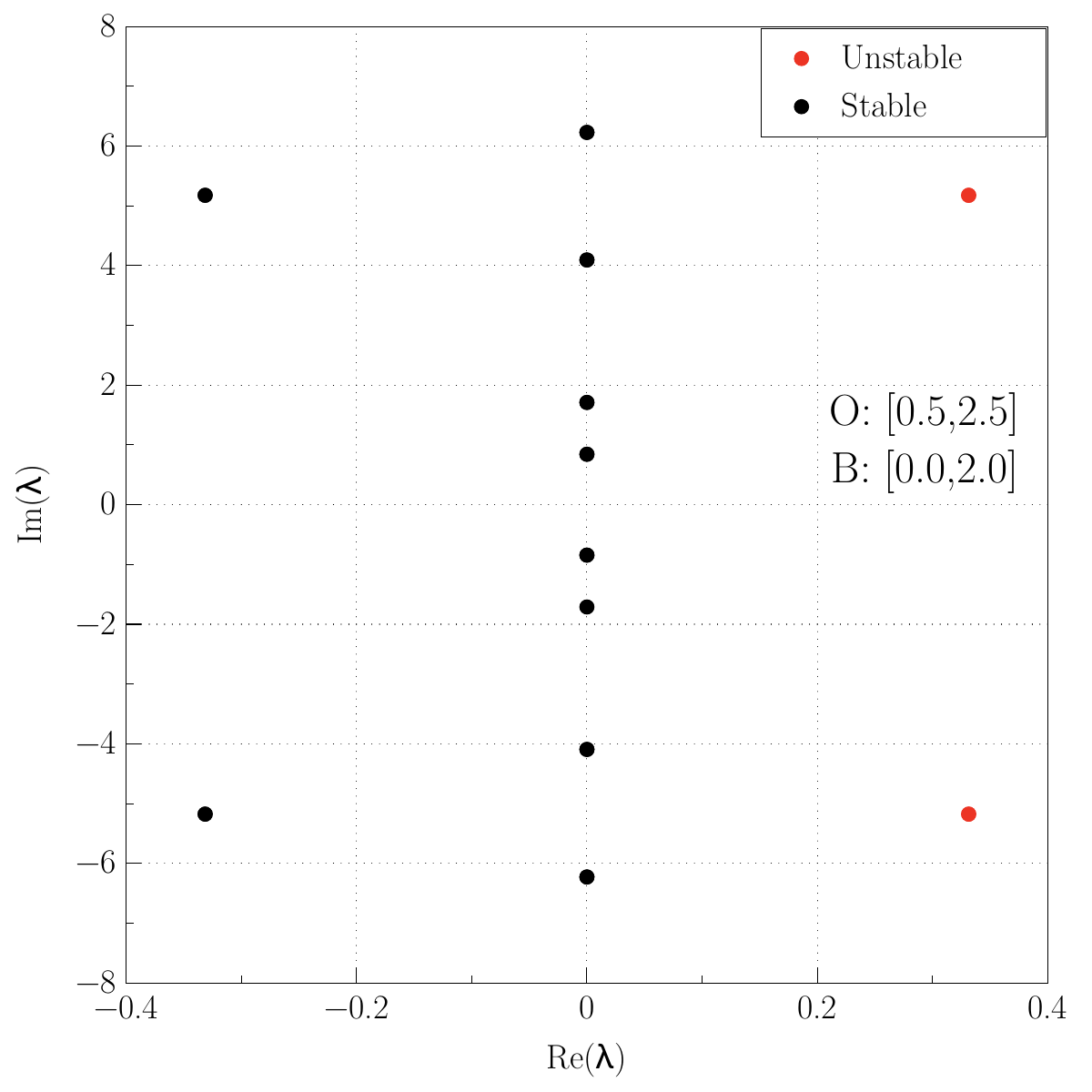}
   \includegraphics[width=1.85in]{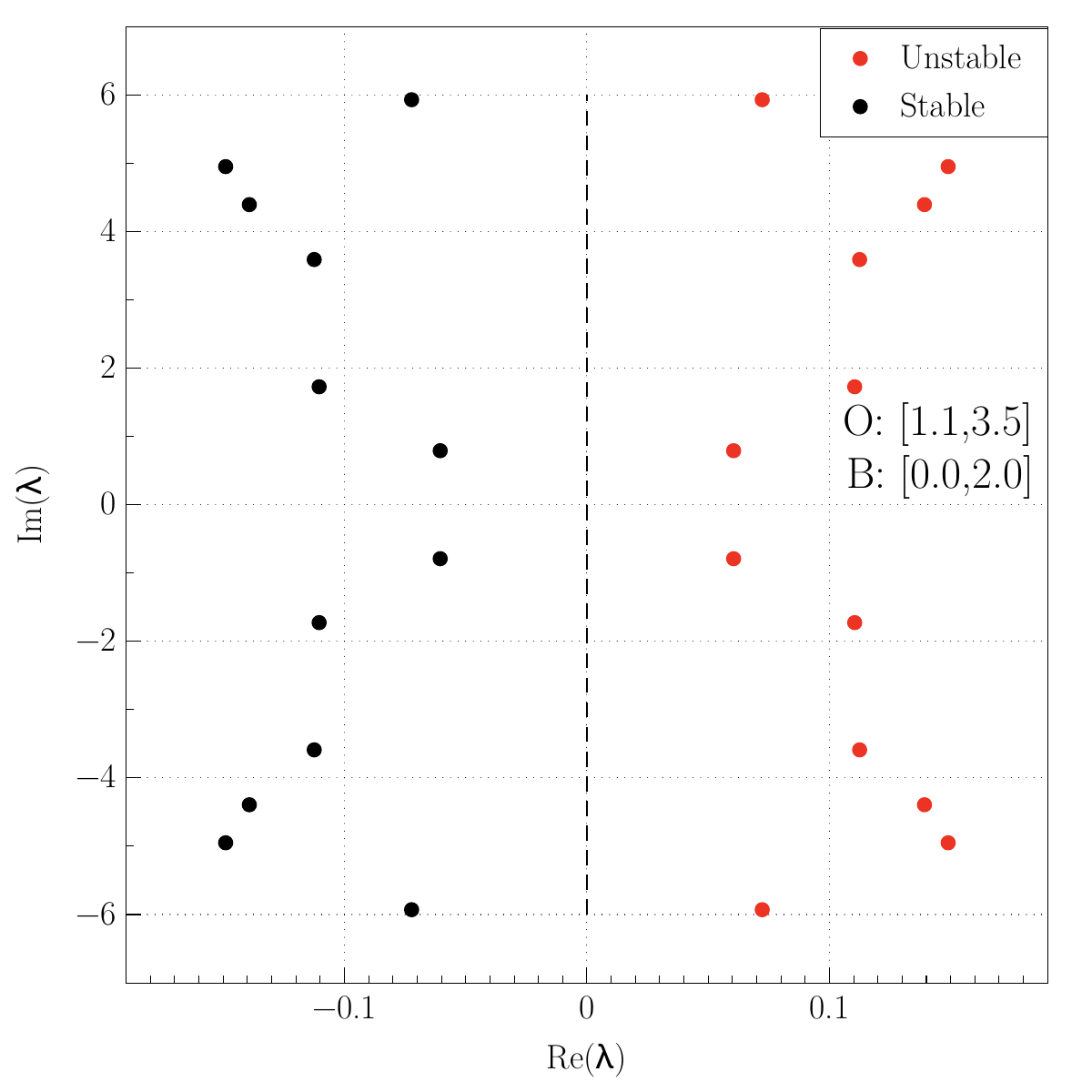}
   \caption{Eigenvalues of the DGSEM approximation of the system \eqref{eq:1DWaveEquationSystem} for the overset problem with one element per subdomain, the central numerical flux, and $N=5$, showing that the approximation can have unstable eigenvalues depending on the size and location of the overset domain relative to the base domain. }
   \label{fig:CentChimeraSystemEvs}
\end{figure}

The fact that the matrix $\mmatrix Q$ is indefinite means that the coupling between the grids is inherently unstable, and therefore it is necessary to add some kind of stabilization. For the DGSEM, dissipation is typically added by using a dissipative numerical flux. For the scalar problem, if, instead of using the central numerical flux, we use the upwind one, $U^*(U^L,U^R) = U^L$, the approximation at a point on a single domain \eqref{eq:PtwiseScalarU}
becomes
\begin{equation}
w_j\dot U_j +\delta_{jN}\oneHalf U_N - \delta_{j0}\left(g_L - \oneHalf U_0   \right) + \oneHalf\left\{ \sum_{n=0}^N U_n\left(w_jD_{jn} - w_nD_{nj} \right)\right\}=0.
\label{eq:BaseWithDiss}
\end{equation}
The matrix system therefore becomes
\begin{equation}
\mmatrix W_e\frac{d}{dt} \vec U =  \oneHalf\mmatrix Q\vec U + \oneHalf \vec g,
\label{eq:SDMatrixEquationWDiss}
\end{equation}
where
\begin{equation}
\mmatrix Q = \mmatrix S - \left[\begin{array}{ccccc}1 &  &  &  &  \\ & 0 &  &  &  \\ &  & \ddots &  &  \\ &  &  & 0 &  \\ &  &  &  & 1\end{array}\right] \equiv \mmatrix S - \mmatrix R
\end{equation}
The matrix $\mmatrix S$ is skew-symmetric, whereas the matrix $\mmatrix R$ is positive semi-definite. The matrix $\mmatrix Q$ is therefore negative semi-definite, noting that for any $\vec x \ne 0$,
\begin{equation}
\vec x^T \mmatrix Q\vec x = \vec x^T \mmatrix S\vec x - \vec x^T \mmatrix R\vec x= - \vec x^T \mmatrix R\vec x = -(x_0^2 + x_N^2)\le 0,
\label{eq:NegSemiDefinite}
\end{equation}
and so its eigenvalues are always stable.

When we move to the overset grid problem for the scalar equation,
the base domain equation stays the same, while the overset grid solution satisfies
\begin{equation}
w_j\dot V_j +\delta_{jN}\oneHalf V_N - \delta_{j0}\left(U(b) - \oneHalf V_0   \right) + \oneHalf\left\{ \sum_{n=0}^N U_n\left(w_jD_{jn} - w_nD_{nj} \right)\right\}=0.
\end{equation}
Then the system of equations in matrix form is
\begin{equation}
\mmatrix W \frac{d}{dt}\left[\begin{array}{c}\vec U \\\vec V\end{array}\right] = \oneHalf \left[\begin{array}{cc}\mmatrix S - \mmatrix R & \mmatrix 0 \\\mmatrix B & \mmatrix S-\mmatrix R\end{array}\right]\left[\begin{array}{c}\vec U \\\vec V\end{array}\right] + \vec g = \oneHalf\mmatrix Q\left[\begin{array}{c}\vec U \\\vec V\end{array}\right] + \vec g,
\end{equation}
where
\begin{equation}
\mmatrix B =2 \left[\begin{array}{ccc}\ell_0(b) & \dots & \ell_N(b) \\0 & \dots & 0 \\\vdots & \ddots & 0 \\0 & \dots & 0\end{array}\right]
\end{equation}
is now the matrix that couples the base grid to the overset grid. Again, the matrix is block lower triangular, so the eigenvalues of $\mmatrix Q$ are those of $\mmatrix S - \mmatrix R \le 0$, and are therefore stable.

The final question is whether or not it is enough to use the upwind numerical flux for the system to guarantee stable eigenvalues.  The answer relies, critically, on the fact that for the constant coefficient problem in one space dimension, the system can be diagonalized to decouple left and right-going waves. The ability to diagonalize the system was also critical in showing that the system problem with characteristic boundary conditions is well-posed \cite{KOPRIVA2022110732}.

For the system we have
\begin{equation}
\begin{split}
w_j\dot{\statevec U}_j &+ \delta_{jN}\left(\statevec F^*(\statevec U_N,\statevec g_R) - \oneHalf \statevec F_N   \right) - \delta_{j0}\left(\statevec F^*(\statevec g_L,\statevec U_0) - \oneHalf \statevec F_0   \right)
\\&+ \oneHalf\left\{ \sum_{n=0}^N \mmatrix A\statevec U_n\left(w_jD_{jn} - w_nD_{nj} \right)\right\}=0,
\end{split}
\end{equation}
where
$\statevec F = \mmatrix A\statevec U = \mmatrix A^+\statevec U + \mmatrix A^-\statevec U$ and $\statevec F^*(\statevec U^L, \statevec U^R) = \mmatrix A^+\statevec U^L  + \mmatrix A^-\statevec U^R$.
Then for $\statevec U$,
\begin{equation}
\begin{split}
w_j\dot{\statevec U}_j &+\delta_{jN}\left(\oneHalf(\mmatrix A^+ +|\mmatrix A^-| )\statevec U_N  + \mmatrix A^-\statevec g_r \right)
+ \delta_{j0}\left(\oneHalf(\mmatrix A^+ +|\mmatrix A^-|)\statevec U_0   \right)
\\&+ \oneHalf\left\{ \sum_{n=0}^N \mmatrix A\statevec U_n\left(w_jD_{jn} - w_nD_{nj} \right)\right\}=0.
\end{split}
\end{equation}
For $\statevec V$,
\begin{equation}
\begin{split}
w_j\dot{\statevec V}_j &+\delta_{jN}\left(\oneHalf(\mmatrix A^+ +|\mmatrix A^-|)\statevec V_N    \right)
- \delta_{j0}\left(\mmatrix A^+\statevec g_L -\oneHalf(|\mmatrix A^-| + \mmatrix A^+)\statevec V_0  \right)
\\&+ \oneHalf\left\{ \sum_{n=0}^N \mmatrix A\statevec V_n\left(w_jD_{jn} - w_nD_{nj} \right)\right\}=0.
\end{split}
\end{equation}
Note that $\mmatrix A^+ +|\mmatrix A^-|  = |\mmatrix A|$.

For the overlap interface points, $\statevec g_R = \statevec V(c)$, $\statevec g_L = \statevec U(b)$ given by \eqref{eq:OverlapPointsSystem}.
Then the system of ODEs is
\begin{equation}
\hat{\mmatrix W} \frac{d}{dt}\left[\begin{array}{c}\vec {\statevec U} \\\vec {\statevec V}\end{array}\right]
= \oneHalf \left[\begin{array}{cc}\mmatrix S - \mmatrix R & \mmatrix O \\\mmatrix B & \mmatrix S - \mmatrix R\end{array}\right]\left[\begin{array}{c}\vec {\statevec U} \\\vec {\statevec V}\end{array}\right] = \oneHalf\mmatrix Q\left[\begin{array}{c}\vec {\statevec U} \\\vec {\statevec V}\end{array}\right]
\end{equation}
where, now,
\begin{equation}
\mmatrix B = 2\left[\begin{array}{ccc}\ell_0(b)\mmatrix A^+ & \dots & \ell_N(b)\mmatrix A^+\\\mmatrix 0 & \dots & \mmatrix 0 \\\vdots & \ddots & \mmatrix 0 \\\mmatrix 0 & \dots & \mmatrix 0\end{array}\right],\quad  \mmatrix O
=
2\left[\begin{array}{ccc}\mmatrix 0 & \dots & \mmatrix 0 \\\vdots & \ddots & \mmatrix 0 \\\mmatrix 0 & \dots & \mmatrix 0 \\\ell_0(c)|\mmatrix A^-| & \dots & \ell_N(c)|\mmatrix A^-|\\\end{array}\right],
\end{equation}
and
\begin{equation}
\mmatrix R = \left[\begin{array}{ccccc}(|\mmatrix A^-| + \mmatrix A^+) &  &  &  &  \\ & 0 &  &  &  \\ &  & \ddots &  &  \\ &  &  & 0 &  \\ &  &  &  & (|\mmatrix A^-| + \mmatrix A^+)\end{array}\right] \ge 0.
\end{equation}

Let
\begin{equation}
\mmatrix {$\mathcal P$} = \mathrm{diag}(\mmatrix P),
\end{equation}
where $\mmatrix P$ is the matrix that diagonalizes $\mmatrix A$, i.e. $\mmatrix P^T\mmatrix A\, \mmatrix P = \mmatrix {$\Lambda$}$, since $\mmatrix A$ is symmetric. We assume that $\mmatrix P$ is arranged so that the matrix of eigenvalues is partitioned as
\begin{equation}
\mmatrix {$\Lambda$} = \mmatrix {$\Lambda$}^+ + \mmatrix {$\Lambda$}^- = \left(\begin{array}{cc}\mmatrix {$\Lambda$}^+_s & \mmatrix 0 \\\mmatrix 0 & \mmatrix 0\end{array}\right) +  \left(\begin{array}{cc}\mmatrix 0 & \mmatrix 0 \\\mmatrix 0 & \mmatrix {$\Lambda$}^-_s\end{array}\right),
\end{equation}
where $\mmatrix {$\Lambda$}^+_s$ is the subset that contains the positive eigenvalues and $\mmatrix {$\Lambda$}^-_s$ contains the negative ones. The transformation also partitions the
solution vector into left and right going characteristic variables,
\begin{equation}
\mmatrix P^T \statevec U = \left(\begin{array}{c}\statevec U^+ \\\statevec U^-\end{array}\right).
\end{equation}
Then the matrix in which the blocks become diagonal within their blocks,
\begin{equation}
{\mmatrix Q}_d \equiv\mmatrix {$\mathcal P$}^T\mmatrix Q\, \mmatrix {$\mathcal P$}
=  \left[\begin{array}{cc}{\mmatrix S}_d - {\mmatrix R}_d & {\mmatrix O}_d \\ {\mmatrix B}_d &  {\mmatrix S}_d-{\mmatrix R}_d\end{array}\right],
\end{equation}
is similar to  $\mmatrix Q$ and therefore has the same eigenvalues. The matrix $ {\mmatrix S}_d$ remains skew-symmetric, while the other matrices become
\begin{equation}
 {\mmatrix B}_d = 2\left[\begin{array}{ccc}\ell_0(b)\mmatrix {$\Lambda$}^+ & \dots & \ell_N(b)\mmatrix {$\Lambda$}^+\\\mmatrix 0 & \dots & \mmatrix 0 \\\vdots & \ddots & \mmatrix 0 \\\mmatrix 0 & \dots & \mmatrix 0\end{array}\right],\quad  {\mmatrix O}_d
=
2\left[\begin{array}{ccc}\mmatrix 0 & \dots & \mmatrix 0 \\\vdots & \ddots & \mmatrix 0 \\\mmatrix 0 & \dots & \mmatrix 0 \\\ell_0(c)|\mmatrix {$\Lambda$}^-| & \dots & \ell_N(c)|\mmatrix {$\Lambda$}^-|\\\end{array}\right],
\end{equation}
and
\begin{equation}
{\mmatrix R}_d = \left[\begin{array}{ccccc}|\mmatrix {$\Lambda$}|  &  &  &  &  \\ & 0 &  &  &  \\ &  & \ddots &  &  \\ &  &  & 0 &  \\ &  &  &  & |\mmatrix {$\Lambda$}|\end{array}\right] \ge 0.
\end{equation}

Since the characteristic polynomial is invariant under row swaps, we can re-order the equations without changing the eigenvalues. Let $\mathbb P$ be a permutation matrix that re-orders the equations to separate the positive and negative eigenvalues of $\mmatrix A$, with the state vector $\mathbb P\mmatrix P^T \statevec U$ now of the form
\begin{equation}
\left(\begin{array}{c}\vec{\statevec U}^+ \\ \vec{\statevec {V}}^+ \\\vec{\statevec {V}}^- \\\vec{\statevec {U}}^-\end{array}\right).
\end{equation}
Then the re-ordered matrix can be written as
\begin{equation}
{\mmatrix Q} _r = \mathbb P\mmatrix Q_d= \left(\begin{array}{cccc}{\mmatrix S}^+-{\mmatrix R}^+& 0 & 0 & 0 \\{\mmatrix B}^+ & {\mmatrix S}^+-{\mmatrix R}^+ & 0 & 0 \\0 & 0 & {\mmatrix S}^--{\mmatrix R}^- & 0 \\0 & 0 & {\mmatrix O}^- & {\mmatrix S}^--{\mmatrix R}^-\end{array}\right)
\end{equation}
where the $\pm$ denotes the submatrices associated with the positve and negative eigenvalues.
The matrix ${\mmatrix Q} _r$ is also lower block triangular (and completely decouples the left- and right-going characteristic variables), and therefore has the eigenvalues of the diagonal blocks, which are each negative semi-definite since $\mmatrix S^\pm$ is skew-symmetric and $\mmatrix R^\pm$ is negative semi-definite, per \eqref{eq:NegSemiDefinite}. Therefore $\mmatrix Q\le 0$, and its eigenvalues are stable.

Unlike when the dissipation-free approximation is used, the eigenvalues of the DGSEM system are stable with the upwind numerical flux. However, this stability relies critically on the ability to diagonalize and sort the system to decouple left- and right-going waves. In general this is only possible in one space dimension, so the result does not extend to multiple space dimensions.


\bibliographystyle{elsarticle-num-names}
\bibliography{ChimeraStability.bib}

\end{document}